\documentclass[11pt]{amsart}

\usepackage{multirow}
\usepackage{mathabx}
\usepackage{amssymb}
\usepackage{amsmath, array}
\usepackage{graphicx,rotating,booktabs}
\usepackage{pdflscape}
\usepackage{leftidx}
\usepackage{tikz-cd} 

\newtheorem{theorem}{Theorem}[section]
\newtheorem{cor}{Corollary}[section]
\newtheorem{prop}{Proposition}[section]
\newtheorem{defn}{Definition}[section]

\newcommand{\Oo}{\mathcal{O}}
\newcommand{\gc}{\mathfrak{g}}
\newcommand{\hc}{\mathfrak{h}}
\newcommand{\pc}{\mathfrak{p}}
\newcommand{\qc}{\mathfrak{q}}

\newcommand{\lc}{\mathfrak{l}}
\newcommand{\kc}{\mathfrak{k}}
\newcommand{\ssc}{\mathfrak{s}}
\newcommand{\nc}{\mathfrak{n}}
\newcommand{\Sl}{\mathfrak{sl}(2)}
\newcommand{\SL}{SL(2, \mathbb{C})}
\newcommand{\la}{\langle}
\newcommand{\ra}{\rangle}
\newcommand{\Ad}{\texttt{Ad}}
\newcommand{\ad}{\texttt{ad}}
\newcommand{\N}{\mathcal{N}}
\newcommand{\W}{\textbf{W}}
\newcommand{\Pp}{\mathcal{P}}
\newcommand{\Aq}{\mathcal{A}}
\newcommand{\Rr}{\mathcal{R}}

\newcommand{\Ii}{\mathcal{I}}
\newcommand{\Jj}{\mathcal{J}}
\newcommand{\Bb}{\mathcal{B}}
\newcommand{\Xx}{\mathcal{X}}
\newcommand{\Zz}{\mathcal{Z}}
\newcommand{\Ss}{\mathcal{S}}
\newcommand{\as}{\alpha}
\newcommand{\bs}{\beta}

\newcommand{\dlam}{\lambda^\vee}
\newcommand{\gr}{\text{gr}}

\newcommand{\GC}{\textbf{G}}
\newcommand{\EG}{^\Gamma\!\GC^\vee}
\newcommand{\ExG}{^\Gamma\!\GC}
\newcommand{\KC}{\textbf{K}}

\newcommand{\GR}{\GC(\mathbb{R})}

\newcommand{\HC}{\textbf{H}}
\newcommand{\HR}{\HC(\mathbb{R})}
\newcommand{\BC}{\textbf{B}}
\newcommand{\LC}{\textbf{L}}
\newcommand{\LR}{\LC(\mathbb{R})}
\newcommand{\PC}{\textbf{P}}
\newcommand{\QC}{\textbf{Q}}

\newcommand{\NC}{\textbf{N}}
\newcommand{\CC}{\textbf{C}}

\newcommand{\AVC}{AV_\mathbb{C}}
\newcommand{\AVR}{AV_\mathbb{R}}
\newcommand{\AVT}{AV_\theta}

\newcommand{\Cc}{\mathcal{C}}
\newcommand{\RWG}{W_\mathbb{R}} 
\newcommand{\KGB}{\texttt{KGB}}

\newcommand{\Int}{\text{Int}}
\newcommand{\Aut}{\text{Aut}}
\newcommand{\Out}{\text{Out}}
\newcommand{\Gal}{\text{Gal}}

\newcommand{\CMc}{MR1251060}
          
\newcommand{\Carter}{MR1266626}
\newcommand{\ABV}{MR1162533}
\newcommand{\AVReal}{MR1168491}
\newcommand{\VoganOrange}{MR908078}

\newcommand{\Algo}{MR2485793}

\newcommand{\Yama}{yamashita1994}

\title{Special Unipotent Arthur Packets for Real Reductive Groups}
\author{Jonathan Fernandes}
\begin{document}
\bibliographystyle{alpha}
\maketitle
\tableofcontents

\section{Introduction}
Suppose $\GC$ is a complex connected algebraic Lie group. Studying the representation theory of the real forms of $\GC$ has been a major focus over past few decades. Fix a real form $\GR$ of $\GC$ and let $\Pi(\GR)$ be the set of irreducible representations of $\GR$. This set has been completely classified by results of Langlands, Harish-Chandra, Vogan, Knapp, Zuckerman etc, and is called the Langlands classification. Let $\Pi_h(\GR))$ be the set of irreducible representations of $\GR$ equipped with an $\GR$-invariant hermitian form, results of Knapp and Zuckerman completely classify this set in terms of a certain condition on the Langlands parameters. Finally, let $\Pi_u(\GR)$ be the set representations of $\GR$ equipped with a positive definite $\GR$-hermitian form. Clearly, $\Pi_u(\GR)\subset \Pi_h(\GR)$, and one would assume that the classification of $\Pi_h(\GR)$ would somehow provide a way to classify $\Pi_u(\GR))$ in $\Pi_h(\GR))$. Classification of $\Pi_u(\GR)$ has been a ambitious process, the closest that we have arrived, is at identifying $\Pi_u(\GR)$ as a subset of $\Pi_h(\GR)$ using the Atlas Software, a product of Atlas of Lie Groups project, a collaboration between a wide network of mathematicians led by Jeffrey Adams, David Vogan, Marc Van Leuven, etc. 
\\

Vogan laid out a program to classify $\Pi_u(\GR)$ in the book \cite{\VoganOrange}, which basically comes down to classifying certain unitary representations, called unipotent representations, of Levi subgroups of $\GR$, and using these to induce up to unitary representations of $\GR$. There are two steps involved in this program: 

\begin{enumerate}
\item given a real reductive group $\HR$, classify the unipotent representations attached to $\HR$.
\item given a real group $\GR$ and a unitary representation $\pi$ of $\GR$, show that $\pi$ is induced from a unipotent representation of some Levi subgroup $\LR$ of $\GR$. 
\end{enumerate}

We will only be concerned with step (1) of this program, in fact we only work out a very special case of this step. There is no unique definition of unipotent representations, in fact Vogan provides at least five different approaches towards defining unipotent representations. We will be concerned with what are known as Arthur's unipotent representations which come in packets correponding to a special unipotent Arthur parameter.  
\\
The goal of this paper is to compute special unipotent Arthur packets. In the process of this computation, we compute a few invariants for representations of $\GR$, called associated varieties. These computations involve understanding representation theory in two different view points, firstly the standard approach via Langlands classification and secondly in terms of the ``Atlas setting" which is suitable for computer implementation. Furthermore, there is fair amount of background involving nilpotent orbits and associated varieties. 
\\

This paper is organized as follows:
\begin{enumerate}
\item we begin with outlining basic definitions and results about nilpotent orbits that are relevant to us. 
\item we recall the Springer correpondence and description of Weyl group representations.
\item we then describe the setting in terms of standard/traditional approach.
\item we introduce the Atlas setting, describe how representations are classified in the setting and the some structure theory of parabolics.
\item we then introduce associated varieties for $\GR$, describe these for some important class of representations. 
\item In Section 6, leads to first important algorithm and computation of theta forms of an even nilpotent orbit. 
\item  In Section 7, we introduce unipotent Arthur representations and packets, it ends with an algorithm to compute these packets in certain cases. 
\item Section 8 provides an application of results proved in Section 6 and 7. 
\end{enumerate}
The expert reader can safely jump right to Section 6 through 8, where lies the bulk of the novelty and the main results and arguments of this paper. 
\\

The results presented in this paper are part of the doctoral thesis work of the author under the supervision of Jeffrey Adams, who suggested the problem and has been a source of immense encouragement and help to the author. We thank Peter Trapa whose notes on real forms of nilpotent orbits motivated the results in Section 6. The author is thankful to David Vogan for laying the foundations for almost everything that we use in this paper. We are also thankful to Michael Rapoport, Dipendra Prasad, Thomas Haines, Xuhua He, Jonathan Rosenberg and Patrick Brosnan for providing insightful feedback on a preliminary version of the results of this paper. Finally, special thanks to the Atlas of Lie Groups project, especially Marc Van Leuwen for adding the scripting functionality into the Atlas software which has played a pivotal role in the development of this paper.

\section{Nilpotent Orbits}
Let $\GC$ be a complex reductive group, with complex Lie algebra $\gc$. Fix a real form $\GR$ of $\GC$ and let $\gc_\mathbb{R}$ be the corresponding real Lie algebra. Let $\KC$ be the complexfication of the maximal compact subgroup of $\GR$ and let $\theta$ be the Cartan involution so that $\GC^\theta = \KC$. Fix $\HC \subset \GC$, a Cartan subgroup and let 

\begin{eqnarray}
X^*(\HC) &=& \{ \text{The lattice of rational characters (into $\mathbb{C}$) of $\HC$.}  \}\\	\notag
 X_*(\HC) &=& \{ \text{The lattice of one parameter subgroups of $\HC$.}  \} 
\end{eqnarray}
so there is a natural pairing
\begin{equation}
\la\; ,\; \ra: X^*(\HC)\times X_*(\HC) \longrightarrow \mathbb{Z}. 
\end{equation}
Using the following natural isomorphisms,
\begin{equation} \label{lattice_isom_cartan}
 \hc \simeq X_*(\HC)\otimes_\mathbb{Z} \mathbb{C}, \quad \hc^*\simeq X^*(\HC)\otimes_\mathbb{Z} \mathbb{C},
\end{equation}
where $\hc^*$ is the vector space dual of $\hc$, we can extend the pairing to 
\begin{equation}\label{bilinear_pairing}
\la\; ,\; \ra: \hc^*\times \hc \longrightarrow \mathbb{C}. 
\end{equation}

Now, fix a set of roots $\Delta(\gc, \hc)$ for $\gc$ and let $\Pi(\gc, \hc)$ be a choice of simple positive roots. Let $\Delta^\vee(\gc, \hc)$ and $\Pi^\vee(\gc, \hc)$ be the corresponding set of coroots and simple coroots. The set of weights for $\GC$ is defined as 
\begin{equation}
 P(\GC) := \{ \lambda \in  X^*(\HC)\otimes_\mathbb{Z} \mathbb{C} : \; \langle \lambda, \as^\vee\rangle \in \mathbb{Z} \;\; \text{for all}\;\; \as \in \Delta \}. 
\end{equation}

Also, the co-weights for $\GC$ are defined as
\begin{equation}
P^\vee (\GC) := \{ \lambda^\vee \in  X_*(\HC)\otimes_\mathbb{Z} \mathbb{C} : \; \langle \as, \lambda^\vee\rangle \in \mathbb{Z} \;\; \text{for all}\;\; \as \in \Delta \}. 
\end{equation}

We can identify $2\pi i \; X_*(\HC)$ with the kernel of the exponential map $\text{exp}: \hc \longrightarrow \HC$, under this identification we have 
\begin{equation}
P^\vee (\GC) = \{ \lambda^\vee \in \hc : \; \text{exp}(2\pi i \;\lambda^\vee)\in Z(\GC)\},
\end{equation}
where $Z(\GC)$ is the center of $\GC$. 
Also,
\begin{equation}
P(\GC) = \{ \lambda \in \hc^* : \; \text{exp}(2\pi i \;\lambda)\in Z(\GC^\vee)\},
\end{equation}
where $Z(\GC^\vee)$ is the center of the complex connected dual group $\GC^\vee$. 
\\

We outline some facts about nilpotent adjoint and coadjoint orbits for $\GC$, additional details can be found in ~\cite{\CMc}. The group $\GC$ acts on $\gc$ via the adjoint action
\begin{equation}
\Ad: \GC\longrightarrow \text{End}(\gc), \quad g\mapsto \Ad(g). 
\end{equation}
An element $X \in \gc$ is called nilpotent if there exist a $k\in \mathbb{N}$ such that $\ad(X)^k=0$. The set of all nilpotent elements in $\gc$ is called the nilpotent cone and is denoted as $\N$. 

\begin{defn}[Nilpotent Orbit]
A nilpotent orbit in $\gc$ is an orbit in $\N$ under the $\Ad$ action of $\GC$.
\end{defn}

If $X\in \gc$ is a nilpotent element, then we write $\Oo_X:=\Ad(G)\cdot X$ for the nilpotent orbit in $\gc$. 

\begin{theorem}[Jacobson-Morozov]\label{J-M1}
Suppose $\gc$ is a complex reductive Lie algebra. Let $X$ be a non-zero nilpotent element in $\gc$. Then, there exist $H\in \hc$ (semisimple) and $Y\in \gc$ (nilpotent) such that
\begin{equation}
[H,X]=2X \quad,  [H,Y]=-2Y \quad \text{and} \quad [X,Y]=H,
\end{equation}
where the bracket is the Lie algebra bracket in $\gc$. 
\end{theorem}
The set $\{X, H, Y\}$ is called a standard $\mathfrak{sl}(2)$-triple, and $X$ is called its nilpositive element. Suppose $\mathcal{A}$ is the set of $\Ad(\GC)$-conjugacy classes of $\mathfrak{sl}(2)$-triples in $\gc$, then, we can define a map
\begin{equation}
\Omega: \mathcal{A}\longrightarrow \{\text{nilpotent orbits}\}\;\; ; \;\; \Omega(\{X,H,Y\})=\Oo_X.
\end{equation} 
The map $\Omega$ is bijective. 
We will conjugate the triple so that the semisimple element $H$ of the triple is dominant with respect to $\Pi(\gc, \hc)$. Furthermore, $H$ can be chosen such that it belongs to $P^\vee(\GC)$. 
\\
Using the bilinear pairing in Equation \ref{bilinear_pairing} we can label the nodes of the Dynkin diagram for $\gc$ by the integer $\langle\alpha, H\rangle$. Such a diagram is called a labelled Dynkin diagram and we denote it by $\mathcal{D}_H$. If $H$ is the semisimple element of a standard $\mathfrak{sl}(2)$-triple, then, using $\mathfrak{sl}(2)$-representation theory one can show that labels for the Dynkin diagram can only be one of either $0, 1$ or $2$. Let $\mathfrak{D}$ be the set of labelled Dynkin diagrams corresponding to standard $\mathfrak{sl}(2)$-triples. 

\begin{defn}[Even Nilpotent Orbits]
Let $\Oo$ be a nilpotent orbit for $\GC$ and let $\{X, Y, H\}$ be the corresponding Jacobson-Morosov triple. We say $\Oo$ is an even nilpotent orbit if any one of the following equivalent conditions hold. 
\begin{enumerate}
\item all the nodes of the labelled Dynkin diagram $\mathcal{D}_H$ are even (i.e. either 0 or 2). 
\item $\frac{1}{2} H \in P^\vee(\GC)$. 
 \end{enumerate}
\end{defn}

When $\gc$ is reductive Lie algebra of classical type, there is a classification of nilpotent orbits in terms of partitions. We refer the reader to [\cite{\CMc}, Theorem 5.1.1-5.1.4] for details of this classification.

\subsection{Induction of Nilpotent Orbits}
Many nilpotent orbits in $\gc$ can be induced from nilpotent orbits on subalgebras of $\gc$. We introduce some ideas (relevant to us) regarding induction of nilpotent orbits. Most of the details can be found in Chapter 7 of \cite{\CMc}.  
\\

Let $\pc=\lc+\nc$ be a parabolic subalgebra in $\gc$. Let $\PC$ be the corresponding parabolic subgroup in $\GC$. Suppose $\Oo_\lc$ is a nilpotent orbit in $\lc$. We have the following result.

\begin{theorem}[\cite{\CMc}, Theorem 7.1.1]
As in the notation above, recall that, $\Ad(\PC)$ is a connected subgroup of $\Ad(\GC)$ with Lie algebra $\pc$. There there is a unique nilpotent orbit $\Oo_\gc$ in $\gc$ meeting $\Oo_\lc+\nc$ in an open dense set. The intersection $\Oo_\gc \cap (\Oo_\lc+\nc)$ consists of a single $\Ad(\PC)$-orbit. The orbit $\Oo_\gc$ above will called the induced orbit from $\Oo_\lc$ and will be denoted as
$$\Oo_\gc=\text{Ind}_\pc^\gc(\Oo_\lc).$$
\end{theorem}

The induced orbit only depends on the Levi factor $\lc$ of $\pc$: 

\begin{theorem}[\cite{\CMc}, Theorem 7.1.3]
Let $\pc=\lc+\nc$ and $\pc'=\lc+\nc'$  be two parabolic subalgebras in $\gc$ have the same Levi subalgebra $\lc$ and let $\Oo_\lc$ be a nilpotent orbit in $\lc$. Then 
$$\text{Ind}_\pc^\gc(\Oo_\lc)=\text{Ind}_{\pc'}^\gc(\Oo_\lc).$$
\end{theorem}

We say a nilpotent orbit is a Richardson orbit if it is induced from the trivial orbit on some parabolic subalgebra in $\gc$. Suppose $\Oo$ is a nilpotent orbit with $\Sl$-triple $\{X, H, Y\}$ and let $\mathcal{D}_H$ be the labelled Dynkin diagram for $\Oo$. Let $\Delta(\Oo)$ be the complement set of vertices labelled 2 in $\mathcal{D}_H$. Let $\lc$ be the Levi subalgebra generated by the roots in $\Delta(\Oo)$.

\begin{theorem}[\cite{\CMc}, Theorem 7.1.6]
Let $\mathcal{D}'(\Oo)$ be the labelled sub-diagram of $\mathcal{D}(\Oo)$ consisting of vertices labelled 0 or 1. If $\mathcal{D}'$ is the labelled Dynkin diagram of a nilpotent orbit $\Oo_\lc$ in $\lc$, then, $\Oo=\text{Ind}_\lc^\gc(\Oo_\lc)$. 
\end{theorem}

Recall that even nilpotent orbits have the nodes of their Dynkin diagram labelled either 0 or 2. If $\Oo$ is even, $\lc = \text{Cent}_{\gc}(H)$ and, $\mathcal{D}'(\Oo)$ defined in the theorem corresponds to the trivial orbit $0_\lc$ in $\lc$. Therefore, we have

\begin{cor}
Suppose $\Oo$ is an even nilpotent orbit in $\gc$. Then, $\Oo$ is a Richardson orbit; induced from the trivial orbit on the Levi subalgebra $\lc$ of $\gc$ generated by the nodes labelled 0 in $\mathcal{D}(\Oo)$. 
\end{cor} 

\subsection{Real Nilpotent Orbits}
Recall that $\N$ was defined to be the cone of nilpotent elements in $\gc$. The real nilpotent cone is defined to be the nilpotents in $\gc_\mathbb{R}$:
\begin{equation}
 \N_\mathbb{R} := \N\cap \gc_\mathbb{R}.
\end{equation}

The real nilpotent cone $ \N_\mathbb{R}$ is a finite union of $\Ad(\GR)$-conjugacy classes. When $\gc$ is of classical type, the conjugacy classes are parameterized by signed young tableau, for more details about this classification and its explicit realization we refer the reader to (\cite{\CMc}, Chapter 9). 

\begin{defn}[Real form of a complex nilpotent orbit]\label{real_form_of_O}
Let $\Oo$ be a complex nilpotent orbit for $\GC$. Let $\GR$ be a real form of $\GC$. By a real form of $\Oo$ we mean a $\GR$-conjugacy class of nilpotent elements in $\Oo\cap\gc_\mathbb{R}$. 
\end{defn} 

We use an alternated description of $\N_{\mathbb{R}}$ based on the Cartan involution $\theta$. Let $\gc = \kc\oplus \ssc$ be the Cartan decomposition of $\gc$ with respect to $\theta$, that is $\kc=\gc^\theta$ and $\ssc = \gc^{-\theta}$.
\\
Let 
\begin{equation}
\N_\theta := \{\text{ Nilpotent elements in $\ssc$}\}.
\end{equation}

Since $\KC$ preserves $\ssc = \gc/\kc$, $\KC$ acts on acts on $\mathfrak{s}$ and this action partitions $\N_\theta$ into finitely many orbits. 

\begin{theorem}[Kostant-Sekiguchi]
There is a natural bijective correspondence between nilpotent $\GR$-orbits in $\gc_\mathbb{R}$ and the nilpotent $\KC$-orbits in $\mathfrak{s}$. 
\end{theorem}  

We now define,
\begin{defn}[$\theta$-form of a complex nilpotent orbit]\label{theta_real_form_of_O}
Let $\Oo$ be a complex nilpotent orbit for $\GC$. Let $\theta$ be the Cartan involution defining the real form $\GR$ of $\GC$. By a $\theta$-form of $\Oo$ we mean a $\KC$-conjugacy class of nilpotent elements in $\Oo\cap\ssc$, where $\KC=\GC^\theta$ and $\gc=\kc\oplus\ssc$ is the Cartan decomposition of $\gc$ with respect to $\theta$. 
\end{defn} 
Since $\theta$-forms are defined using Cartan involutions, they are better suited for our applications.

\subsection{Coadjoint Nilpotent Orbits}
In applications, nilpotent orbits arise in the dual vector space $\gc^*$ of $\gc$. Note that $\gc^*$ does not have a Lie algebra structure, and as such, there is no direct way of making sense of nilpotent elements in $\gc^*$. If $\gc$ is a complex reductive Lie algebra, one can define an invariant non-degenerate symmetric bilinear form on $\gc$:
\begin{equation}
\langle \cdot, \cdot \rangle_\gc: \gc \times \gc \longrightarrow \mathbb{C}
\end{equation}
 such that $ \langle \cdot, \cdot \rangle_\gc$ restricted to $[\gc, \gc]$ is a nonzero constant multiple of the Killing form on $[\gc, \gc]$. 
 \\
 
The fact that $ \langle \cdot, \cdot \rangle_\gc$ is non-degenerate implies that the map $\phi: \gc \longrightarrow\gc^*$ defined by
\begin{equation}
X \mapsto \phi_X:=\langle X, \cdot \rangle_\gc \in \gc^*,
\end{equation}
is an isomorphism of vector spaces. Define the nilpotent cone in $\gc^*$ as $ \N^* :=  \phi(\N)$. 

Suppose $\Oo$ is a nilpotent orbit in $\gc$ with nilpositive element $X\in \Oo$ (so that $\Oo = \Ad(G)\cdot X$), then, we define the corresponding coadjoint orbit to be 
\begin{equation}
\Oo^*:= \Ad(G)\cdot \phi_X \subset \N^*.
\end{equation}

We can use the map $\phi$ to identify the other coadjoint cones of nilpotent elements with respect to a real form $\GR$ of $\GC$ corresponding to a Cartan involution $\theta$ as follows:
\begin{equation}
\N_\mathbb{R}^* :=  \phi(\N_\mathbb{R}),  \quad \N_\theta^* :=  \phi(\N_\theta).
\end{equation}
 
In this setting we have bijections:
\begin{enumerate}
	\item $\N/\GC$ and $\N^*/\GC$.
	\item $\N_\theta/\KC$ and $\N^*_\theta/\KC$.
	\item $\N_\mathbb{R}/\GR$ and $\N^*_\mathbb{R}/\GR$.
\end{enumerate}
Therefore using the Kostant-Sekiguchi correspondence $\N_\theta/\KC, \N^*_\theta/\KC, \N_\mathbb{R}/\GR$ and $\N^*_\mathbb{R}/\GR$ are all in bijective correspondence.
\subsection{Duality of Nilpotent Orbits}
Let $\GC$ be a complex reductive group and $\GC^\vee$ be the corresponding dual group. 
\\
When $\GC$ is of classical type, there is a basic duality due to Spaltenstein defined as a map $d : \N(\GC) \longrightarrow \N(\GC^\vee)$, called the duality map. We refer the reader to Section 6.3 in \cite{\CMc} for explicit description in terms of partitions.

\subsection{The Springer Correspondence}
We recall the Springer correspondence. Define $\mathcal{B}$ to be set of Borel subalgebras in $\gc$.Given a nilpotent element $X$, the variety $\mathcal{B}_X$ is the set of Borel subalgebras containing $X$. The group $\GC^X:=\text{Cent}_{\GC}(X)$ acts on $\mathcal{B}_X$ via the adjoint action. The induced action of this action on the cohomology $H^*(\mathcal{B}_X, \mathbb{C})$ is trivial on $\GC_0^X$ so that $A(\Oo_X):= \GC^X/\GC^X_0$ acts on $H^*(\mathcal{B}_X, \mathbb{C})$. 
\\
For an irreducible representation $(\pi, V_\pi)$ of $A(\Oo_X)$ define
\begin{equation}
H^*(\mathcal{B}_X, \mathbb{C})_\pi:= \text{Hom}_{A(\Oo_X)}(V_\pi, H^*(\mathcal{B}_X, \mathbb{C})).
\end{equation}
We are now ready to state the Springer correspondence:

\begin{theorem}[Springer]
For any nilpotent element $X$, there is a natural action of $\W$ on $H^*(\mathcal{B}_X, \mathbb{C})$.
\begin{enumerate}
\item The actions of $\W$ and $A(\Oo_X)$ commute; so $\W$ acts on $H^*(\mathcal{B}_X, \mathbb{C})_\pi$ for $\pi \in \widehat{A(\Oo_X)}$.
\item The natural maps $$ H^*(\mathcal{B}, \mathbb{C})\longrightarrow H^*(\mathcal{B}_X, \mathbb{C}),$$ induced by $H^*(\mathcal{B}_X, \mathcal{B})$, are $\W$ - equivariant. 
\item For $\pi\in \widehat{A(\Oo_X)}$, the representation $\sigma(X, \pi)$ of $\W$ on $H^{\text{dim}_\mathbb{R}(\mathcal{B}_X)}(\mathcal{B}_X, \mathbb{C})_\pi$ is irreducible or zero.
\item If $\pi$ is trivial, $\sigma(X, \pi)\neq 0$.
\item Suppose $\sigma \in \widehat{W}$. Then there are:  a nilpotent element $X\in \gc$, unique upto $\Ad(G)$; and a unique $\pi\in \widehat{A(\Oo_X)}$, such that $$\sigma=\sigma(X, \pi).$$
\end{enumerate}
\end{theorem}
The correspondence 
\begin{equation}
(G\cdot X, \pi)\longrightarrow \sigma(X, \pi)
\end{equation}
is called the Springer correspondence. We write 

\begin{equation}
\sigma(\Oo_X)=\sigma(\Oo_X,1).
\end{equation}

The Springer correspondence provides a way of attaching to each nilpotent orbit $\Oo$ a finite set of $\W$-representations, having a distinguished element $\sigma(\Oo)$. 
\subsection{Weyl Group Representations in Classical Type}
We go over some facts about Weyl group representations in types $B_l$ and $C_l$, details of the general situation can be found in \cite{\Carter}.  

\begin{theorem}[Irreducible Weyl group representations of Type $B_l$ and $C_l$]
The irreducible representations of the Weyl group $\W(B_l)$ (and $C_l$) are in bijection with pairs of partitions $(\alpha, \beta)$ such that $|\alpha|+|\beta|=l$. We write $\sigma_{(\as, \bs)}$ for the $\W$-representation corresponding to $(\as, \bs)$. 
\end{theorem}

Suppose $\as = (\as_0, \as_1, \dots, \as_m) $ and $\bs=(\bs_0, \bs_1, \dots, \bs_{m-1})$ (we allow for the parts to be zero, requiring that $\as$ has one more part that $\bs$) such that
\begin{equation}
0\le \as_0\le \as_1 \le \dots \le \as_m \quad \text{and} \quad 0\le \bs_0\le \bs_1 \le \dots \le \bs_{m-1}
\end{equation}
Lusztig attaches to $(\as, \bs)$ the following symbol:

$$
\begin{pmatrix}
\lambda \\ \mu
\end{pmatrix}:=
\begin{pmatrix}
\lambda_0& & \lambda_1 &&  \dots && \lambda_{m-1}&& \lambda_m \\
&\mu_0 &&\mu_1&& \dots && \mu_{m-1}&
\end{pmatrix}
$$
where $\lambda_i=\as_i+i$ and $\mu_j=\bs_i+i$ for $i=0, 1, 2, \dots$. 
\begin{defn}[Special $\W$-representation]
Let $\sigma_{(\as, \bs)}$ be a irreducible representation of $\W$ with correspoding symbol $\begin{pmatrix}
\lambda\\\mu
\end{pmatrix}$. We say $\sigma_{(\as, \bs)}$ is special if 
\begin{equation}
\lambda_0\le \mu_0 \le \lambda_1\le \mu_1\le \lambda_2\le \mu_2 \le \dots \le \lambda_{m+1}.
\end{equation}
\end{defn}

\begin{defn} [Special Nilpotent Orbit]
 Let $\Oo$ be a nilpotent orbit in $\gc$. We say $\Oo$ is special if the Weyl group representation $\sigma(\Oo)$ associated to $\Oo$ via the Springer correspondence is a special $\W$-representation.
\end{defn}
The Springer correspondence is explicitly realized when $\GC$ is of classical type, we make use of this realization in our computations and to this end has been implemented into Atlas. For the interested reader, the algorithm can be found in (\cite{\Carter}, pages 419-423).

\section{An Overview of the setting}
We outline the basic setting that we use for rest of this paper. Let $\GC$ be a complex connected reductive algebraic group. Let $\Int(\GC), \Aut(\GC)$ and $\Out(\GC)$ be the groups of inner automorphisms, automorphisms and outer automorphisms respectively of $\GC$. We have the following exact sequence
\\
\begin{center}
$\begin{tikzcd}
 1\arrow{r}  & \Int(\GC) \arrow{r}  & \Aut(\GC) \arrow{r}{p}  & \Out(\GC) \arrow{r}  &1,
 \end{tikzcd}
 $
\end{center}
so that $\Out(\GC) \simeq \Aut(\GC)/\Int(\GC)$. 
\\
A \textit{splitting datum} for $\GC$ is a tuple $(\BC, \HC, \{X_\as\})$, where $\BC$ is a Borel subgroup of $\GC$, $\HC$ a Cartan subgroup and $\{X_\as\}$ is the set of root vectors for the of simple roots of $\HC$ in $\BC$. An involution of $\GC$ is said to be \textit{distinguished} if it preserves a splitting datum. 
\\
Let $\GC^\vee$ be the dual group of $\GC$. There is a bijection between $\Out(\GC)$ and $\Out(\GC^\vee)$ (Definition 2.11, \cite{\Algo}) and we denote it by $\gamma\in \Out(\GC) \mapsto \gamma^\vee\in \Out(\GC^\vee)$. 
\\
Fix $\gamma\in \Out(\GC)$ an element of order two. An involution $\theta\in\Aut(\GC)$ is said to be in the \textit{inner class} of $\gamma$ if $p(\theta)=\gamma$. We will say that two involutions $\theta$ and $\theta'$ are inner to each other if they have the same image in $\Out(\GC)$ under the map $p$. We call the pair $(\GC, \gamma)$ a basic data, and the corresponding dual basic data is given by $(\GC^\vee, \gamma^\vee)$. Let $\Gamma=\Gal(\mathbb{C}/\mathbb{R}) = \{1, \sigma\}$, then in this setting,

\begin{defn}[Extended group, $L$-group]
The extended group for the pair $(\GC, \gamma)$ is the semidirect product $\ExG := \GC \rtimes \Gamma$, where $\sigma\in \Gamma$ acts by the distinguished involution in $p^{-1}(\gamma)$. 
The $L$-group for the pair $(\GC, \gamma)$ is defined to be the extended group for the pair $(\GC^\vee, \gamma^\vee)$, often denoted as $\GC^L$ or just $\EG$. 
\end{defn}

A \textit{real form} of $\GC$ is an antiholomorphic involutive automorphism $\sigma: \GC \longrightarrow \GC$. Let $\sigma_c$ be the compact real form of $\GC$ chosen such that it commutes with $\sigma$, then $\theta = \sigma\circ\sigma_c$ is a holomorphic involution of $\GC$. We prefer to work with holomorphic maps, and, to this end we need the following, (Theorem 3.2, \cite{\Algo}),

\begin{theorem}
The map $\sigma \mapsto \theta$ gives a bijection between $\GC$-conjugacy classes of antiholomorphic involutions and $\GC$-conjugacy classes of holomorphic involutions of $\GC$. 
\end{theorem}

A \textit{Cartan involution} of $\GC$ is a holomorphic involution of $\GC$. Henceforth, by a real form of $\GC$ we will mean a $\GC$-conjugacy class of Cartan involutions.

\begin{defn}[Strong real form]
A strong involution of $\GC$ in the inner class defined by $\gamma$ is an element $\xi\in \ExG-\GC$ satisfying $\xi^2 \in Z(\GC)$. The set of strong involutions is denoted by $\Ii(\GC, \gamma)$. 
\\ 
A strong real form of $\GC$ in the inner class of $\gamma$ is the $\GC$-conjugacy class of a strong involution.
\end{defn}
Given a strong real form $\xi\in \Ii(\GC, \gamma)$, we can define a Cartan involution of $\GC$ as $\theta_\xi=\Int(\xi)$, $\KC_\xi=\text{Stab}_\GC(\theta_\xi) = \GC^{\theta_\xi}$. There is a surjective map from $\Ii(\GC, \gamma)/\GC$ onto the set of all real forms of $\GC$ in the class defined by $\gamma$, this map is bijective if $\GC$ is adjoint. 

\subsection{Langlands and Arthur Parameters}
Let $\RWG$ be the Weil group of $\mathbb{R}$. A \textit{Langlands parameter} is a homomorphism $\phi: W_\mathbb{R} \longrightarrow \;\EG$ such that the following diagram of $L$-morphisms commutes 
\begin{center}
$\begin{tikzcd}
 W_\mathbb{R} \arrow{r}{\phi}  \arrow{rd}{} 
  &  ^\Gamma\GC^\vee \arrow{d}{} \\
    & \Gamma
\end{tikzcd}$
\end{center}
and, $\phi(\mathbb{C}^\times)$ is contained in the set of semisimple elements of $\GC^\vee$. The group $\GC^\vee$ acts on such parameters by conjugation. For any conjugacy class of such parameters is attached a $L$-packet of representations of real forms in the innner class defined by $\gamma$. 
\\
Using (\cite{\ABV}, Proposition 5.6), one can identify the set of Langlands parameters with pairs $(y, \dlam)$ satisfying the conditions
\begin{enumerate}
\item $y\in \EG - \GC^\vee$ and $\dlam \in \hc^\vee$ is a semisimple element,
\item $y^2 = \text{exp}(2\pi i\dlam)$, and
\item $[\dlam, \Ad(y)\dlam] = 0$. 
\end{enumerate}

An \textit{Arthur parameter} for $\GC$ is a homorphism $\psi: \RWG \times SL(2, \mathbb{C}) \longrightarrow \EG$ satisfying 
\begin{enumerate}
\item the restriction of $\psi$ to $\RWG$ is a tempered (Definition 22.3, \cite{\ABV}) Langlands parameter,
\item the restriction of $\psi$ to $SL(2, \mathbb{C})$ is holomorphic. 
\end{enumerate}
The group $\GC^\vee$ acts on such parameters by conjugation. 
\\
We say that $\psi$ is a \textit{unipotent Arthur parameter} if $\psi$ restricted to the identity component of $\RWG$ is trivial. Given an Arthur parameter $\psi$, define the Langlands parameter $\phi_\psi$ to be
\begin{equation}
\phi_\psi: \RWG \longrightarrow \GC^\vee \quad \phi_\psi(w):= \psi(w, \begin{pmatrix} |w|^{1/2}&0 \\ 0 & |w|^{-1/2}\end{pmatrix}.
\end{equation}
Associated to $\psi$ is an Arthur Packet of representations of real forms in the inner class given be $\gamma$ (Definition 22.6, \cite{\ABV}) containing the $L$-packet corresponding to $\phi_\psi$ and at most finitely many additional representations of of strong real forms in the given inner class. One of the main results of this paper is to devise and implement an algorithm to compute these packets when $\psi$ is assumed to be unipotent. 

\section{Atlas of Liegroups requisites}
We will make use of the Atlas of Lie groups setting. More details can be found in \cite{\Algo} and in resources available at www.liegroups.org. 
\\
We continue in the setting of the previous section. We fix a real form $\GR$ of $\GC$, wich corresponding Cartan involution $\theta$, so that $\KC=\GC^\theta$. Furthermore, fix a pinning  $(\BC,\HC,\{X_\as\}$ for $\GC$. We use the Harish-Chandra homorphism to associate to $\lambda \in \hc^*$, an infinitesimal character which we will also denote by $\lambda$, which only depends on the integral Weyl group $\W_\lambda$ and is unique if dominant with respect to a fixed choice of simple positive roots $\Pi(\gc, \hc)$. We say that $\lambda$ is regular (resp. integral) if $\la \lambda, \as^\vee\ra \neq 0$ (resp. $\in \mathbb{Z}$) for roots $\as\in \Pi(\gc, \hc)$. 
\\
It is well known that irreducibel admisible representations of $\GR$ are parameterized by irreducible admissible $(\gc, \KC)$-modules.
We define the following sets: 

\begin{center}
	$
	\begin{array} {c l l}
	\mathfrak{M} (\gc, \KC) &=& \{ \text{Category of finite length $(\gc, \KC)$-modules.}  \}\\
		\mathbb{K}\mathfrak{M} (\gc, \KC) &=& \{ \text{Grothendieck  group of $\mathfrak{M}(\gc, \KC)$.}  \}\\

	\mathfrak{M} (\gc, \KC, \lambda) &=& \{ \text{Category of finite length $(\gc, \KC)$-modules with inf char $\lambda$.}  \} \\
	
	\mathbb{K}\mathfrak{M} (\gc, \KC, \lambda) &=& \{ \text{Grothendieck group of $\mathfrak{M}(\gc, \KC, \lambda)$.}  \}\\
	\Pi(\gc, \KC) &=& \{ \text{Equiv classes of irred admissible $(\gc, \KC)$-modules.}  \}\\
	\Pi(\gc, \KC, \lambda) &=& \{ \text{$J\in\Pi(\gc, \KC)$ such that  inf char of $J$ is $\lambda$.}  \}

	\end{array}
	$
\end{center}

By results of Harish-Chandra, the set $\Pi(\gc, \KC, \lambda)$ is a finite set altough there is no closed form formula for its cardinality. It is desirable to find a combinatorial description of $\Pi(\gc, \KC, \lambda)$, and the results in \cite{\Algo} do exactly that. We now outline the basic components of this description.
\\

Recall that $\Ii(\GC, \gamma)$ is the set of strong real forms for the basic data $(\GC, \gamma)$, we will denote it as $\Ii$ when there is no confusion about the basic data in question. The \textit{one-sided parameter space} is the set
\begin{equation}
\Xx(\GC, \gamma) := \{\xi \in \Ii\;| \; \xi\in \text{Norm}_{\GC^\Gamma-\GC}(\HC)\}/\HC,
\end{equation}
the equivalence is via conjugation. When there is no confusion about the basic data, we will denote $\Xx(\GC, \gamma)$ by just $\Xx$. Given a dual basic data $(\GC^\vee, \gamma^\vee)$, we will denote $\Xx(\GC^\vee, \gamma^\vee)$ as just $\Xx^\vee$. Furthermore, given $x\in \Xx$, we denote the fiber of $x$ to be 
\begin{equation}
\Xx[x] := \{ x'\in \Xx\;| \; q^{-1}(x)\; \text{ is $\GC$-conjugate to}\; q^{-1}(x')\},
\end{equation}
where $q$ is the natural projection map. 
Fix $x\in \Xx$ and suppose $\xi\in\Ii$ is such that $q(\xi)=x$, we will denote $\theta_{x, \HC}$ to be the Cartan involution $\theta_\xi$ restricted to $\HC$. The \textit{two-sided parameter space} is defined as 
\begin{equation}
\Zz(\GC, \gamma) := \{(x, y)\in \Xx\times \Xx^\vee\;| \; (\theta_{x, \HC})^t = -\theta_{y, \HC^\vee}\} \subset \Xx\times \Xx^\vee.
\end{equation}

The following theorem provides the combinatorial setup that we want,
\begin{theorem}[Adams-DuCloux, \cite{\Algo}, Thm 10.3]\label{Atlas_main_thm}
 Fix a set $\Lambda \subset P_{\text{reg}}(\GC, \HC)$ of representatives of $P(\GC)/X^*(\HC)$. Let $I = \Ii/\GC$ be a set of representatives for the strong real forms of $\EG$. For each $\xi_i\in I$ let $\theta_{\xi_i}$ be the Cartan involution corresponding to conjugation by $\xi_i$ and let $\KC_{\xi_i}$ be the fixed points of $\theta_{\xi_i}$. There is a natural bijection 
 \begin{equation}
 \Zz(\EG) \leftrightarrow \coprod_{\xi_i\in I} \coprod_{\lambda_j\in \Lambda} \Pi(\gc, \KC_{\xi_i}, \lambda_j).
 \end{equation}
\end{theorem}

By (Corollary 9.9 \cite{\Algo}), the set $\Xx$ is in bijection with the disjoint union of $\KC_\xi$-orbits on $\GC/\BC$ (denoted as $\KC_\xi\backslash \GC/\BC$) as $\xi$ varies over $\Ii$. If we fix a strong real form $\xi$ such that $q(\xi)=x$, then there is a bijection between $\Xx[x]$ and $\KC_\xi\backslash \GC/\BC$. Fix an infinitesimal character $\lambda$ for $\GC$ then the set $\Pi(\gc, \KC_\xi, \lambda)$ satisfies
\begin{equation}
\Pi(\gc, \KC_\xi, \lambda) \subset \Zz(\GC, \gamma) \subset \Xx[x]\times\Xx^\vee\simeq \KC_\xi\backslash \GC/\BC \times \left( \coprod_{\eta^\vee_j\in \Ii^\vee} \KC_{\eta_j^\vee}^\vee\backslash \GC^\vee/\BC^\vee\right),
\end{equation}
where $\Ii^\vee$ is a set of representatives for the strong real forms of $\EG$. 
\\

Now fix $\eta^\vee \in \Ii^\vee$ and let $y = q(\eta^\vee)\in \Xx^\vee$. We assume that $\lambda$ satisfies $(\eta^\vee)^2 = \text{exp}(2\pi i \lambda)$, then we have the following definition 

\begin{prop}[Block of $(\gc, \KC_\xi)$-modules]
A block of irreducible $(\gc, \KC_\xi)$-modules at infinitesimal character $\lambda$, where $\lambda$ satisfies $(\eta^\vee)^2 = \text{exp}(2\pi i\lambda^\vee)$, is the set of irreducible $(\gc, \KC_\xi)$-modules corresponding to the parameters in the set $\Bb(\xi, \eta^\vee, \lambda^\vee)$  with the following combinatorial description:

\begin{eqnarray}
\Bb(\xi, \eta^\vee, \lambda^\vee) &=& \left(\Xx[x] \times \Xx^\vee[y]\right)\cap \Zz(\GC, \gamma)\\\notag
													&=& \left(\KC_\xi\backslash \GC/\BC \times  \KC_{\eta^\vee}^\vee\backslash \GC^\vee/\BC^\vee\right) \cap \Zz(\GC, \gamma),
\end{eqnarray}
\end{prop}

If we choose $\lambda^\vee \in \Lambda^\vee\subset P^\vee(\GC, \HC)$ such that $\xi^2 = \text{exp}(2\pi i \lambda)$, we can define the dual block of $\Bb(\xi, \eta^\vee, \lambda^\vee) $, as the block of irreducible  $(\gc^\vee, \KC^\vee_{\eta^\vee})$-modules at infinitesimal character $\lambda^\vee$, as follows 

\begin{eqnarray}
\Bb^\vee(\xi,\eta^\vee, \lambda) &=& \Bb(\eta^\vee, \xi, \lambda^\vee) \\\notag
													&=&\left(\Xx[y] \times \Xx^\vee[x]\right)\cap \Zz(\GC\vee, \gamma^\vee)\\\notag
													&=& \left(\KC_{\eta^\vee}^\vee\backslash \GC^\vee/\BC^\vee \times \KC_\xi\backslash \GC/\BC\right)\cap \Zz(\GC^\vee, \gamma^\vee),
\end{eqnarray}
with the same compatibility condition as in the definition of $\Bb$. 
\\

In this setting, we can realize Vogan duality as follows:
\begin{defn}[Vogan Duality]
Vogan duality is the natural bijection between the sets $\Bb(\xi, \eta^\vee, \lambda^\vee)$ and $\Bb^\vee(\xi, \eta^\vee, \lambda^\vee)$ obtained by the map $(x, y) \mapsto (y, x)$.  
\end{defn}
Vogan duality provides a bijection $\pi\leftrightarrow \pi^\vee$ between irreducible representations in blocks $\Bb$ and $\Bb^\vee$, this plays a crucial role in our computations.
\\

We end this section with a brief description of $L$-packets for $\GC$. The $L$-packets for $\GC$ are parameterized by $\KGB$-elements for $\GC^\vee$. If we fix a $\KGB$-element $\texttt{y}_0=p(\eta^\vee)$ for $\GC^\vee$, the corresponding $L$-packet containing representations of real forms of $\GC$ is given by 
\begin{equation}
\Pi(\GC, y_0) := \{(x, y)\in \Xx[x] \times \Xx^\vee[y_0]\;| \;\; (\theta_{x}|_{\hc})^T = -(\theta_{y_0}|_{\hc^\vee})\}.
\end{equation}
Let $x_0:=p(\xi)$, then $L$-packet of $(\gc, \KC_\xi)$-modules corresponding to the strong real form $\xi$ of $\GC$ and a fixed $y_0$ is given as 
\begin{equation}
\Pi(\gc, \KC_\xi, y_0) := \{(x, y)\in \Pi(\GC, y_0) ;| \; x \sim x_0 \}.
\end{equation}

\subsection{Parabolic Subgroups in Atlas}
The Atlas of Liegroups software computes the set $\Xx[x]$ on computer. We now explain how Atlas computes $\KC$-conjugacy classes of Borel and parabolic subgroups.
\\

Given $\xi \in\Ii$, a strong involution of $\GC$ such that $q(\xi) = x$, the map $\KC_\xi g \BC \mapsto b\BC g^{-1}$ is a bijection between $\KC_\xi\backslash \GC/\BC$ and the $\KC_\xi :=\GC^{\theta_\xi}$ conjugacy classes of  Borel subgroups. Consider the set 
$\Jj:=\{(\xi, \BC')\;|\; q(\xi)=x, \BC' \;\text{a Borel subgroup of $\GC$}\}$. Now fix $\xi_0$ such that $q(\xi_0) \in \Xx[x]$
\\
Cosider the following maps:
\begin{center}
$\begin{tikzcd}
  \Jj \arrow{r}{\phi_1}  \arrow{rd}{\phi_2} 
  &  \Xx[x] \arrow{d}{\psi} \\
    & \KC_\xi\backslash \GC/\BC
\end{tikzcd}$
\end{center}
where for $(\xi, \BC')\in \Jj$, we define $\phi_1$ as follows: choose $g\in \GC$ such that $g\BC' g^{-1} = \BC$ and $g\xi g^{-1} \in \text{Norm}(\HC)$, and define $\phi_1(\xi, \BC') = q(g\xi g^{-1}) \in \Xx[x]$. 
\\
 Define $\phi_2$ as follows: choose $g\in \GC$ such that $g \xi g^{-1} = \xi_0$ and  define $\phi_2(\xi, \BC')$ to be the $\KC_\xi$-conjugacy class of $g\BC' g^{-1}$. 
\\
Both $\phi_1$ and $\phi_2$ are bijections and hence induce a bijection 
\begin{equation}
\psi = \phi_2\circ \phi_1^{-1}: \Xx[x] \longrightarrow \KC_\xi\backslash \GC/\BC
\end{equation}

We can generalize the above construction to parabolic subgroups. Let $S\subset \Pi(\gc, \hc)$, and let $\PC_S$ be the standard parabolic in $\GC$ defined by $S$. All parabolics conjugate to $\PC_S$ will be called parabolics of type $S$. Furthermore, the Weyl group $\W$ acts naturally on $\Xx[x]$ and so does the group $\W_S$ generated by the simple roots in $S$. In this setting, we have the followig picture:

\begin{center}
$\begin{tikzcd}
  \{(\xi, \Pp)\;|\; q(\xi)\in \Xx[x], \PC \; \text{parabolic of type-$S$}\} \arrow{r}{\phi_1}  \arrow{rd}{\phi_2} 
  &  \Xx[x]/ \W_S\arrow{d}{\psi} \\
    & \KC_\xi\backslash \GC/\PC_S
\end{tikzcd}$
\end{center}
where given $(\xi, \PC)$, we define $\phi_1$ as follows: choose $g\in \GC$ such that $g\PC g^{-1} = \PC_S$ and $g\xi g^{-1} \in \text{Norm}(\HC)$, and define $\phi_1(\xi, \PC) = q(g\xi g^{-1}) \in \Xx[x]$. 
\\
 Define $\phi_2$ as follows: choose $g\in \GC$ such that $g \xi g^{-1} = \xi_0$ and  define $\phi_2(\xi, \PC)$ to be the $\KC_\xi$-conjugacy class of $g\PC g^{-1}$. 
\\
Both $\phi_1$ and $\phi_2$ are bijections and hence induce a bijection 
\begin{equation}
\psi = \phi_2\circ \phi_1^{-1}: \Xx[x]/\W_S \longrightarrow \KC_\xi\backslash \GC/\PC_S
\end{equation}

The main results of this paper use the explicit computation of $\KC_\xi$ conjugacy classes of parabolic subgroups, which can now be done using $\Xx[x]/\sim_S$, the latter computation being implemented in the Atlas software. 
 
Define a finite set $\Pp$ as follows
\begin{equation}
\Pp = \{(S, y)\;|\; S\subset \Pi(\gc, \hc), y\in \Xx[x]\}/\sim,
\end{equation}
where $(S, y)\sim (S',y')$ if and only if $S=S'$ and $y\sim_S y'$. 
\\
For $(S, y)\in \Pp$, let $[Q(S, y)]$ be the $\KC_{\xi_0}$-conjugacy class of parabolic subgroups defined by $\psi(y)$, and let $Q(S, y)$ be a representative parabolic of type $S$ in this class.  

\begin{prop}
The parabolic $Q(S, y)$ is $\theta_x$-stable if and only if $\theta_x(S)=S$. 
\end{prop}

Recall that given a semisimple element $\lambda\in \hc$, let $S(\lambda)\subset \Pi(\gc, \hc)$ be the set of simple roots vanishing on $\lambda$. One can construct a parabolic subalgebra in $\gc$ as follows:
\begin{eqnarray}
\notag
\gc_\as &=& \{ X\in \gc \;\;|\; \texttt{ad}(X)(Y) = \as(X) Y, \text{for all} \; Y\in \gc.\}\\\notag
\mathfrak{n}(\lambda) &=& \sum_{\as\in \Delta(\gc, \hc), \langle\as, \lambda\rangle >0} \gc_\as \\\notag
\lc(\lambda) &=& \text{Cent}_{\hc}(\lambda) = \hc + \sum_{\as\in \Delta(\gc, \hc), \langle\as, \lambda\rangle = 0} \gc_\as \\\notag
\pc(\lambda) &=& \lc(\lambda) + \nc(\lambda).\\\notag
\end{eqnarray}
Let $\PC(\lambda)$ be the parabolic subgroup in $\GC$ corresponding to $\pc_\lambda$, then $\PC(\lambda)$ is a parabolic subgoup of type $S$. Therefore there exists a $y\in \Xx[x]$ such that the parabolic $(\xi, \PC(\lambda) ) \leftrightarrow Q(S(\lambda), y) :=\psi(y)$.  Let $[Q(S(\lambda), y)]$ be the $\KC_\xi$-conjugacy class of parabolic subgroups, and $Q(S, y)$ a representative parabolic of type $S$ in this class.  

\begin{prop}
The parabolic $Q(S(\lambda), y)$ is $\theta_x$-stable if and only if $\theta_x(\lambda)=\lambda$. 
\end{prop}
For more details, the interested reader is can visit www.liegroups.org/Papers. In the later parts of the paper we will be in the ``Atlas Setting" and we define what we mean by that now.

\begin{defn}[Atlas Setting]\label{AtlasSetting}
Let $(\GC, \gamma)$ be a basic data and $(\GC^\vee, \gamma^\vee)$ be the corresponding dual basic data. Let $(\BC, \HC, \{X_\as\})$ be a fixed pinning for $\GC$. Let $\xi$ be a strong real form of $\GC$ in the inner class of $\gamma$ and let $\eta^\vee$ be a strong real form for $\GC^\vee$ in the dual inner class given by $\gamma^\vee$. Corresponding to $\xi$ and $\eta^\vee$, let $\theta_\xi = \text{Int}(\xi)$, and $\theta_{\eta^\vee} = \text{Int}(\eta^\vee)$ be Cartan involutions of $\GC$ and $\GC^\vee$ respectively with maximal complex subgroups $\KC_\xi$ and $\KC^\vee_{\eta^\vee}$ respectively. Finally, let $\lambda$ be an integral infinitesimal character for $\GC$ and let $\Bb(\xi, \eta^\vee, \lambda)$ be the block of irreducible $(\gc, \KC_\xi)$-modules at infinitesimal character $\lambda$ specified by the pair of strong real forms $(\xi, \eta^\vee)$. Also, $\Bb^\vee = \Bb(\eta^\vee, \xi,  \lambda^\vee)$ is the corresponding dual block of irreducible $(\gc^\vee, \KC_{\eta^\vee}^\vee)$-modules.
\end{defn}

\section{Associated Varieties}
\subsection{The Complex Associated Variety}
Let $(\pi, V)$ be an irreducible $(\gc, \KC)$-module. Using the universal property of $U(\gc)$, the universal enveloping algebra of $\gc$, we can think of $(\pi, V)$ as a $(U(\gc), \KC)$ module. Let $I(\pi)$ be the annhilator of $\pi$ in $U(\gc)$, that is 
\begin{equation}
 I(\pi) =  \{ X \in U(\gc) : \pi(X)(v) = 0 \;\;\text{for all}\; v\in V \}. 
\end{equation}

The ideal $I(\pi)$ is a primitive ideal in $U(\gc)$, and one can construct a filtration $\{ I_n(\pi) := U_n(\gc)I(\pi) \}$, where $\{U_n(\gc)\}$ is the standard filtration for $U(\gc)$, and, use it to define the associated graded module:
\begin{equation}
 \text{gr}I(\pi) =  \bigoplus_{n=0}^\infty I_{n}(\pi)/I_{n-1}(\pi).
\end{equation}

Since $U_m(\gc)I_n(\pi)\subset U_{m+n}(\gc)$, $\gr(\pi)$ is a graded ideal in $\gr U(\gc)\simeq S(\gc)$. Using the Poincare-Birkoff-Witt theorem, $\gr U(\gc)\simeq S(\gc)$, and hence we can compute the support of $\gr I(\pi)$. We call the latter the complex associated variety, $AV_\mathbb{C}(\pi)$, of $\pi$, that is: 
\begin{equation}
\AVC (\pi) = \text{Supp}(\gr I(\pi)) = \{\lambda \in \gc^* : X(\lambda) = 0 \;\; \text{for all} \; X\in \gr I (\pi)\}.
\end{equation}

Since $\gr I(\pi)$ is a graded ideal in $S(\gc)$, $\AVC(\pi)$ is a cone in $\gc^*$. Furthermore, it can be shown that $\gr I(\pi)$ must contain some power of the augmentation ideal $J$ of $S(\gc)$, which is  is the collection of $\Ad(G)$- invariant polynomials without constant term. Let $J^k \subset \gr I(\pi)$ for some $k\in \mathbb{N}$. This immediately implies that 
\begin{equation}
 \AVC \subset \text{Supp}(J^k) = \text{Supp}(J).
\end{equation}

The following theorem due to Kostant describes $\text{Supp}(J)$ in $\gc^*$,

\begin{theorem}[\cite{\AVReal}, Theorem 5.7]\label{Var_of_aug_idea}
Suppose $G$ is a reductive Lie group, and $J\subset S(\gc)$, the augmentation ideal. Then the associated variety of $J$ is the cone $\N^*$ of nilpotent elements in $\gc^*$. 
\end{theorem}

An application of the above theorem shows that $\AVC(\pi) \subset \N^*$, the nilpotent cone in $\gc^*$. Therefore $\AVC(\pi)$ must be the closure of a finite union of nilpotent orbits in $\gc^*$, since $\N$ has finite number of orbits. In fact, a much stronger statement is true,

\begin{theorem}[Borho, Brylinski, Joseph]
Let $\GC$ be a complex connected reductive Lie group and let $\GR$ be a real form of $\GC$. Suppose $(\pi, V)$ is an irreducible $(\gc, \KC)$-module, then $\AVC (\pi)$ is the closure of a single nilpotent orbit $\Oo$ in $\gc^*$. 
\end{theorem}

Given $(\pi, V)$, an admissible irreducible $(\gc, \KC)$-module of $\GR$, it is desirable to know if one can compute the invariant $\AVC(\pi)$. In the case when $\GR$ is a classical connected reductive Lie group, we use an algorithm due to Noel and Jackson coupled with the Springer correspondence. These computations has been implemented in the Atlas software.

\subsection{The Real  and the Theta Associated Variety}
We assume that $(\pi, V)$ is a finite length $(\gc, \KC)$-module,  in this case one can show that $V$ is generated by a finite-dimensional subspace $S$ as a $(\gc, \KC)$-module. Using the universal property of $U(\gc)$ we can show that $(\pi, V)$ is a $(U(\gc), \KC)$ - module. Furthermore, for $v\in V$,  $\pi$ satisfies the following conditions:
\vspace{.1 in}
\begin{enumerate}
\item $\text{d} \pi (Z) v = Z\cdot v$ \quad for all \;$Z\in \kc$. 
\item $\pi(k)(X\cdot v) = (\Ad (k)X)\cdot \pi(k) v $ \quad for all \; $k \in \KC$ \; and \; $X\in U(\gc)$. 
\end{enumerate}
\vspace{.1 in}

Using the local finiteness of the action of $\KC$ we can find a finite-dimensional $\KC$-invariant subspace $V_0$ of $V$ that contains $S$. An easy argument shows that $V = U(\gc)V_0$. Therefore, we can construct a filtration $\{ V_n(\pi) := U_n(\gc) V_0\}$ for $V$. Since, $V_0$ was $K$ invariant, and the fact that the action of $\KC$ is compatible with the action of $U(\gc)$, we note that $V_n(\pi)$ is $\KC$-invariant for all $n$. 
\\

Note that $U_n(\gc) V_m(\pi) \subset V_{m+n}(\pi)$. Therefore, this gives us a $\KC$-invariant graded submodule of $\gr U(\gc)$ given by:
\begin{equation}
\text{gr} \;  V =  \bigoplus_{n=0}^\infty V_{n}(\pi)/V_{n-1}(\pi).
\end{equation}

A consequence of the PBW-Theorem is that $\gr V$ is a $(S(\gc), \KC)$ module so that the $S(\gc)$ and the $\KC$ action satisfy the following:
\begin{equation}
\pi(k)(X\cdot v) = (\Ad (k)X)\cdot \pi(k) v\;\; \text{for all}\; k \in \KC, v\in \gr V, \;\text{and}\; X\in S(\gc).
\end{equation}

Differentiating the above equation and noting that $S(\gc)$ is an abelian Lie algebra, we see that 
\begin{equation}
 Z\cdot  v = 0, \quad \text{for all}\;\; X\in \mathfrak{k} \;\;\text{and}\; v\in \gr V.
\end{equation}

As a result, the action $(S(\gc), \KC)$ action on  $\gr  V$ descends to a $(S(\gc/\mathfrak{k}), \KC)$ action. We define
\begin{eqnarray}
 \AVT (\pi) &=& \text{Supp}(\gr V) \\\notag
 				&=& \{\lambda \in (\gc/\mathfrak{k})^* : v\cdot (\lambda) = 0 \;\; \text{for all} \; v\in \gr V \}\subset  (\gc/\mathfrak{k})^*.
\end{eqnarray} 

As in the case of $\AVC (\pi)$, we can show that $\AVT (\pi)$ is closed under dilations, which implies that $\AVT (\pi)$ lies in a cone in $(\gc/\mathfrak{k})^*$. Furthermore, the fact that the module $(\pi, V)$ is quasisimple implies that $\gr V$ contains some power of the augmentation ideal $J$ of $S(\gc/\mathfrak{k})$, therefore
\begin{equation}
\AVT(\pi)\subset \text{Supp}(J).
\end{equation}

As in the case of $\AVC(\pi)$, we can show that $\text{Supp}(J)$ is in  fact the nilpotent cone $\N_\theta^*$ in $(\gc/\mathfrak{k})^*$. Hence $\AVT (\pi)$ must be a union of finitely many $\KC$-orbits in $(\gc/\mathfrak{k})^*$. Using the Kostant-Sekiguchi correspondence, to the $\KC$-orbits in $\AVT(\pi)$ one can find the corresponding $\GR$ orbits in $\N_\mathbb{R}^*$ and the union of these orbits is called the real associated variety of $\pi$, denoted as $\AVR(\pi)$. 
\\

To summarize the above discussion, $\AVT(\pi)$ satisfies:
\begin{equation}
\AVT (\pi) = \overline{\Oo_K^1 \cup \Oo_K^2 \cup \cdots \cup \Oo_K^r},
\end{equation}
where $\Oo_K^i$ are nilpotent $\KC$-orbits on $(\gc/\mathfrak{k})^*$. For $i = 1, 2, \dots, r$, if $\Oo_\mathbb{R}^i$ is the $\GR$ orbit corresponding to $ \Oo_K^i$ under the Kostant-Sekiguchi correspondence then  we define the real associated variety to be
\begin{equation}
\AVR(\pi) = \overline{\Oo_\mathbb{R}^1 \cup \Oo_\mathbb{R}^2 \cup \cdots \cup \Oo_\mathbb{R}^r}.
\end{equation}

The two invariants $\AVC(\pi)$ and $\AVT(\pi)$ attached to $(\pi, V)$ are related as follows:  
 
\begin{enumerate}
\item $\AVC(\pi)\subset \N^*$.
\item $\AVT(\pi) \subset \N_\theta^*$.
\item $\AVR(\pi) \subset \N_\mathbb{R}^*$.
\item If $ \AVT (\pi) = \overline{\Oo_K^1 \cup \Oo_K^2 \cup \cdots \cup \Oo_K^r}$, then 
$$ \AVC(\pi) =  \overline{\Ad(\GC) \cdot \Oo_K^i}.$$
\item If $ \AVR (\pi) = \overline{\Oo_\mathbb{R}^1 \cup \Oo_\mathbb{R}^2 \cup \cdots \cup \Oo_\mathbb{R}^r}$, then 
$$ \AVC(\pi) = \overline{\Ad(\GC) \cdot \Oo_\mathbb{R}^i}.$$
\end{enumerate} 

We end this section with a brief description of cohomologically induced modules and their associated varieties. 
Fix a Cartan involution $\theta$ for $\GC$. Let $\pc = \lc +\nc$ be a theta-stable parabolic subalgebra of $\gc$, so that $\pc$, $\lc$ and $\nc$ are preserved by $\theta$. Let $\gc = \kc\oplus \ssc$, be the Cartan decomposition of $\gc$ and $\textbf{s}= \text{dim}(\ssc \cap \nc)$. We start with a $(\lc, \LC\cap \KC)$- module and construct a $(\gc, \KC)$-module using Zuckerman's cohomological induction functor.
\\

Suppose $Z$ is a one dimensional $(\lc, (\LC\cap \KC))$-module with infinitesimal character $\gamma_L$. 
We can extend $Z$ to a $(\pc, \LC\cap \KC)$-module by making $\nc$ act trivially. Then Zuckerman defines the following produced module
\begin{equation}
  X = \text{pro}_{(\pc, \LC\cap\KC)}^{(\gc, \LC\cap \KC)}(Z),
\end{equation}
and a functor
\begin{equation}
(\mathcal{R}_{\pc, \LC})^0 (Z) = \Gamma_{(\gc, \LC\cap \KC)}^{(\gc, \KC)}(X).
\end{equation}
$\mathcal{R}^0$ is a left exact functor and because the category of $(\lc, (\LC\cap\KC))$-modules has enough injectives, one can define
\begin{equation}
(\mathcal{R}_{\pc, \LC})^i = \text{$i$th right derived functor of $ (\mathcal{R}_{\pc, \LC})^0 $}.
\end{equation}
In this setting,
\begin{theorem}[Zuckerman, Vogan, Theorem 6.8, \cite{\VoganOrange}]
Suppose $\LR$ is a Levi subgroup of $\GR$ attached to the $\theta$-stable parabolic subalgebra $\pc = \lc+ \nc$. Let $\LC$ be the complexification of $\LR$, and $\textbf{s}$ the dimension of $\nc\cap \kc$. Let $2\rho(\nc)$ be the sum of roots positive on $\nc$.  Consider the functors
\begin{equation}
\mathcal{R}^j =  (\mathcal{R}_{\pc, \LC})^j \quad j\in \{0, 1, 2, \dots, \textbf{s}\}
\end{equation}
from the category of $(\lc, \LC\cap\KC )$-modules, to the category of $(\gc, \KC)$-modules. Let $Z$ be a $(\lc, \LC\cap\KC)$-module and let $\hc$ be a Cartan subalgebra of $\lc$. Assume that $Z$ has $\LC$-infinitesimal character $\gamma_L\in \hc^*$ then  
\begin{enumerate}
\item $\mathcal{R}^j(Z)$ has $\GC$-infinitesimal character $\gamma_L + \rho(\nc)$. 
\item Assume that for each root $\as$ of $\hc$ in $\nc$, 
$$ \text{Re}\langle \gamma_L + \rho(n) , \as \rangle \ge 0.$$
Then $\mathcal{R}^j(Z)$ is zero for $j$ not equal to $\textbf{s}$.  
\item Under the above hypothesis, if $Z$ is unitary, then so is $\mathcal{R}^\textbf{s}(Z)$. 
\item If we assume that for each root $\as$  of $\hc$ in $\nc$,
$$ \text{Re}\langle \gamma + \rho(\nc), \as \rangle > 0.$$
Then, if $Z$ is non-zero, so is $\mathcal{R}^\textbf{s}(Z)$. 
\end{enumerate}
\end{theorem}

For most of our applications, we will take $Z$ to be a one dimensional $(\lc, \LC\cap\KC)$-module. Let $\lambda = dZ \in \hc^*$. Let $\pc:=\pc(\lambda) = \lc(\lambda)+\nc(\lambda)$, and by $\Aq_\pc(Z)$, we will really mean $\mathcal{R}^\textbf{s}(Z)$. The modules $\Aq_\pc(Z)$, often denoted as $\Aq_\pc(\lambda)$  are defined using the $\theta$-stable data $(\pc, \lambda:=dZ)$ and have infinitesimal character $\lambda+\rho(\nc(\lambda))$. 
\\
The following theorem shows that even though cohomological induction functor depends on $Z$, the associated variety of the cohomologically induced module $\Aq_\pc(Z)$ depends only on $n(\lambda)$, which depends only on $dZ$. 

\begin{theorem}\cite{\Yama}\label{Aq_AV}
Let $\GR$ be a real group corresponding to the Cartan involution $\theta$. Let $\pi = \Aq_\pc(Z)$, where the $\theta$-stable data is given by $\pc=\pc(\lambda)$ and $Z$ is a one dimensional representation of $\lc$ satisfying $dZ = \lambda\in \lc^*$. Suppose $\langle \lambda + \rho(\nc), \as\rangle >0 $ for all $\as \in \Delta (\nc)$, where $ \Delta(\nc)$ is the set of roots on $\nc$ and $2\rho(\nc)$ is the sum of roots in $\Delta(\nc)$. Let $\gc = \kc\oplus \ssc$ be the Cartan decomposition of $\gc$. 
Then, $\AVT(\pi) = \overline{\KC \cdot (\nc\cap \ssc)}$, is the closure of a single $\KC$-orbit in $\N_\theta$. 
\end{theorem}



\subsection{Coherent Continuation and Translation Functors}
We know that  $X^*(\HC)$ is the lattice of weights of finite dimensional representations for $\GC$. So that given a finite dimensional representation $F$ of $\GC$, the set $\Delta(F)$ of weights of $F$, is a subset of $X^*(\HC)$. We beging with the definition of coherent family:

\begin{defn}[Coherent Family]
A coherent family of virtual modules is a map 
\begin{equation}
\Theta : X^*(\HC)  \longrightarrow \mathbb{K}\mathfrak{M}(\gc, \KC),
\end{equation}
such that:
\begin{enumerate}
\item $\Theta(\lambda)$ has infinitesimal character $\lambda \in X^*(\HC)$; and 
\item For every finite-dimensional representation $F$ of $\GC$, 
\begin{equation}
F\otimes\Theta(\lambda) = \sum_{\mu \in \Delta(F)} \Theta (\lambda + \mu).
\end{equation}
\end{enumerate}
\end{defn} 
Now, fix $\gamma \in X^*(\HC)$ and assume that it is integral. Given $M\in\mathfrak{M}(\gc, \KC, \gamma)$, we say $\Theta : X^*(\HC)  \longrightarrow \mathbb{K}\mathfrak{M}(\gc, \KC)$ is a coherent family through $M$ if $\Theta(\gamma) = M$. The set of all coherent families on $X^*$ is a finite rank, free $\mathbb{Z}$-module. If $\gamma$ is assumed to be regular then we have a basis for coherent families on $X^*(\HC)$ given by $\{ \Theta_M\}$, where $\Theta_M$ is a coherent family through $M$, and $M \in \Pi(\gc, \KC, \gamma)$.
\\

Suppose $w\in \W$ and $\Theta : X^*\longrightarrow \mathbb{K}\mathfrak{M}(\gc, \KC)$ is a coherent family. We can construct a new coherent family $w\cdot \Theta$ defined by 

\begin{equation}
w\cdot \Theta (\lambda) = \Theta( w^{-1}\lambda)
\end{equation}

Since the infinitesimal character is equivalent up to the action of $\W$, $w\cdot \Theta (\lambda)$ has infinitesimal character $\lambda$. Since the weights of a finite dimensional representation of $\GC$ are invariant under the action of $\W$, the second condition for the definition of coherent family is also true for $w\cdot \Theta$. To summarize the above discussion:

\begin{theorem}[Coherent Continuation Action]
Suppose $\gamma$ is a fixed regular integral infinitesimal character. Then there is an action of $\W$ on the set of all coherent families 
$$\Theta: X^* \longrightarrow \mathbb{K}\mathfrak{M}(\gc, \KC)$$ defined by 
\begin{equation}
 w\cdot \Theta (\lambda) = \Theta (w^{-1}\lambda).
\end{equation}
\end{theorem}

We can use the coherent continuation action to define a $\W$ action on $ \mathbb{K}\mathfrak{M}(\gc, \KC, \gamma)$. Since $\Pi(\gc, \KC, \gamma)$ is a basis for $ \mathbb{K}\mathfrak{M}(\gc, \KC, \gamma)$, we only need to define the action of $\W$ on this basis and then linearly extend this action. 
\\
Suppose $J \in \Pi(\gc, \KC, \gamma)$. Choose a coherent family $\Theta:  X^* \longrightarrow \mathbb{K}\mathfrak{M}(\gc, \KC)$ such that $\Theta(\gamma) = J$. Then, 

\begin{equation}
w\cdot J = (w\cdot\Theta)(\gamma).
\end{equation}

We can use the action of $\W$ on  $\mathbb{K}\mathfrak{M}(\gc, \KC, \gamma)$ to define a partial order on representations in $\Pi(\gc, \KC, \gamma)$ as follows:

\begin{defn}
Suppose $X, Y \in \mathfrak{M}(\gc, \KC, \gamma)$. 
\begin{enumerate}
 \item We say $X<_\gamma Y$ if $Y$ appears in $w\cdot X$ for some $w \in \W_\gamma$
 \item We say $X\sim_\gamma Y$ if $X <_\gamma Y$ and $Y<_\gamma X$. 
 \end{enumerate}
\end{defn}
The relation $\sim_\gamma$ is an equivalence relation on $\Pi(\gc, \KC, \gamma)$, and, the equivalence classes are called Harish-Chandra cells.
\\

The following result provides a relation between coherent continuation action and the operation of computing associated varieties.
\begin{prop}
Let $J\in \Pi(\gc, \KC, \gamma)$. Suppose $\Theta: X^* \longrightarrow \mathbb{K}\mathfrak{M}(\gc, \KC)$ such that $\Theta(\gamma) = J$. Let $w\in \W$ be an arbitrary element in $\W_\gamma$, then
\begin{enumerate}
\item $\AVC(J) = \AVC( w \cdot J)$.
\item $\AVT(J) = \AVT( w \cdot J)$.
\item $\AVR(J) = \AVR( w \cdot J)$.
\end{enumerate}
\end{prop}
\begin{proof}
This results comes down to checking that the graded algebras involved in the computations of the associated varieties for $J$ and $w\cdot J$ are all isomorphic, since $J$ and $w\cdot J$ differ only up to tensoring with finite dimensional representations of $\GC$. The conclusion about associated varieties then follows. 
\end{proof}

Recall that $\Pi(\gc, \KC, \gamma)$ is the set of irreducble representations with infinitesimal character $\gamma$. Zuckerman's ideas of tensoring representations with finite-dimensional representations lead to the the theory of translation functors, which is a way of studying the representation theory at an infinitesimal character $\delta$ (possibly different from $\gamma$) in terms of the representation theory at $\gamma$. These ideas will be used extensively in computing unipotent representations. 
\\

Let $\gamma \in \hc^*$ be a fixed infinitesimal character. Fix a weight $\phi \in X^*(\HC)$ and let $F_\phi$ be the finite dimensional representation of $\GC$ with highest weight $\phi$. Let $\pi \in \Pi(\gc, \KC, \gamma)$ be an irreducible $(\gc, \KC)$-module. For, $\gamma \in X^*$ , let $\xi_\gamma: \mathfrak{z}(U(\gc)) \longrightarrow \mathbb{C}$ be the character on $\mathfrak{z}(U(\gc))$ given by Harish Chandra's isomorphism. 
Define the projection map:
\begin{equation}
P_\gamma: \mathbb{K}\mathfrak{M}(\gc, \KC) \longrightarrow \mathbb{K}\mathfrak{M}(\gc, \KC, \gamma),
\end{equation}
where the map takes $\pi \in \mathfrak{M}(\gc, \KC)$  to the largest submodule of $\pi$ annhilated by $(I - \xi_\gamma)|_{\mathfrak{z}(U(\gc))}$. In other words, $P_\gamma$ takes $\pi$ to the largest submodule with infinitesimal character $\gamma$. 

\begin{defn}[Translation (to the Wall) Functor]
Suppose $F_\phi$ is a finite dimensional representation of $\GC$ with highest weight $\phi$. Let $\gamma\in \hc^*$ be regular and integral, and, let $\pi \in \mathbb{K}\mathfrak{M}(\gc, \KC, \gamma)$.  Assume that $\gamma + \phi$ is dominant (possibly singular). The translation functor is the functor
\begin{equation}
T_\gamma^{\gamma + \phi} : \mathfrak{M}(\gc, \KC, \gamma) \longrightarrow \mathfrak{M}(\gc, \KC, \gamma+\phi), \quad \pi \mapsto P_{\gamma+\phi}(\pi\otimes F_\phi).
\end{equation}
\end{defn}

Alternately, we can define translation functors using coherent families as follows: suppose $J\in \Pi(\gc, \KC, \gamma)$, choose a coherent family $\Theta$ such that $\Theta(\gamma) = J$, then
\begin{equation}
T_\gamma^{\gamma + \phi} (J) = \Theta(\gamma+\phi).
\end{equation} 
Since $J\in \Pi(\gc, \KC, \gamma)$ is a basis for  $\mathfrak{M}(\gc, \KC, \gamma)$, we can then linearly extend this definition. Using the relationship of coherent families and associated varieties, we have 
\begin{prop}
Let $\gamma \in \hc^*$ be a regular integral infinitesimal character and let $\pi \in \mathfrak{M}(\gc, \KC, \gamma)$. Let $\phi\in X^*$ be an extremal weight of $F_\phi$, a finite dimensional representation of $\GC$. Then, 
\begin{enumerate}
\item $\AVC(\pi) = \AVC(T_\gamma^{\gamma + \phi} (\pi))$.
\item $\AVT(\pi) = \AVT( T_\gamma^{\gamma + \phi} (\pi))$.
\item $\AVR(\pi) = \AVR( T_\gamma^{\gamma + \phi} (\pi))$.
\end{enumerate}
\end{prop}
\begin{proof}
Since translation functors are nothing but evaluation of coherent families the result follows from the fact that asociated varieties are constant for a fixed coherent family. 
\end{proof}

We package this information about associated varieties being constant on coherent families into the following result,
\begin{prop} \label{AV_of_Cell}
Suppose $$ \mathfrak{M}(\gc, K, \gamma) = \coprod _{\text{HC-Cells}} \Cc, $$
then  $\AVC(\pi), \AVT(\pi)$, and $\AVR(\pi)$ remain constant as one varies $\pi$ over a fixed cell $\Cc$.
\end{prop}

Proposition \ref{AV_of_Cell} allows us to define the notion of associated variety of a cell, that is, if $\Cc$ is a $\W$-cell, we can define $\AVC(\Cc)$, $\AVT(\Cc)$, $\AVR(\Cc)$ to be the respective associated varieties of a fixed $\pi \in \Cc$.
\\

Suppose $\Cc$ is a $HC$-cell. Taking the irreducible representations in  $\Cc$ as a basis, we can linearly extend the coherent continuation action to a $\# \Cc :=c $-dimensional complex representation of $\W$. Understanding this Weyl group representation on the cell $\Cc$ will be the main goal of the following section. 
\\

The coherent continuation action on $\Cc$ contains a unique special Weyl group representation. We can then use the Springer-correspondence to attach a complex nilpotent orbit of $\gc$ to $\Cc$. This complex nilpotent orbit turns out to be the complex associated variety of representations in this cell. There are atleast two approaches to computing the special $HC$ -representation of the cell $\Cc$ - one due to Noel and Jackson, and the other due to Binegar.  In the case when the group $\GR$ is of classical type, the algorithm due to Noel and Jackson is very amenable to implementation in Atlas. 

\subsection{The Noel-Jackson Algorithm} 
Let $\GR$ be the real form a complex classical connect reductive algebraic group $\GC$. The special $\W$-representation attached to a $\W$-cell $\Cc$ can be studied using the sign representation. More precisely, suppose $\pi$ is a representation of $\W$ and let $\textbf{L}(\pi)$ be the set of all parabolic subroups $\Pp$ of $\W$, such that $\text{Res}_\Pp^\W(\pi)$ contains the sign representation of $\Pp$, $\textbf{L}(\pi)$ is called the Levi-set of $\pi$.

\begin{theorem}[Noel-Jackson]
Suppose $\W$ is a Weyl of classical type. Let $\pi$ be an irreducible representation of $\W$. Then, $\pi$ is determined by its Levi set  $\textbf{L}(\pi)$. 
\\
Alternately, starting with a Levi set $\textbf{L}$, it is possible to construct a $\pi \in \widehat{\W}$ such that $\textbf{L}(\pi) = \textbf{L}.$
\end{theorem}

We need the following definition:
\begin{defn}[Tau-invariant]
Suppose $J \in \Pi(\gc, K, \gamma)$. Fix a set of positive simple roots $\Delta^+(\gc, \mathfrak{h})$. We say that a simple root $\alpha$ is in the tau invariant of $J$ if and only if $s_\alpha \cdot J = - J$ ( in the Grothendieck group $\mathbb{K}\mathfrak{M}_\gamma(\gc, K) $).  We denote this set by $\tau(J)$. 
\end{defn}

Fix a $\W$-cell $\Cc$, let $J\in \Cc$ and let $\pi$ be the special $\W$-representation attached to $\Cc$. We can construct a parabolic subgroup $\Pp_J$ of $W$ using the $s_\alpha$ for $\as\in \tau(J)$ as generators. Furthermore, by definition of the tau-invariant, we see that $\text{Res}^\W_{\Pp_J}(\pi)$ contains the sign representation of $\Pp_J$. Therefore, using tau-invariants of representations in $\Cc$, we can extract a Levi-set $\textbf{L}(\Cc)$ for the cell $\Cc$. This is the Levi-set for the coherent continuation representation on $\Cc$. We can now use the Noel-Jackson algorithm to compute the special cell representation on $\Cc$. The paper with the details of this algorithm is in the publication pipeline. With what was available to the authors, we have managed to implement it in to the Atlas software sucessfully. 

\section{Parameterizing Theta Forms of Even Complex Nilpotent Orbits}
Let $\GC$ be a complex connected reductive algebraic group. Let $\EG$ be a $L$-group for $\GC$. We will be in the Atlas Setting (refer \ref{AtlasSetting}) for the rest of this paper.  
We outline an algorithm to parameterize real forms of a even complex nilpotent orbit. 

\subsection{Unipotent Arthur Parameters} 
Fix a unipotent Arthur parameter, say $\psi$. Using the restriction of $\psi$ to $SL(2, \mathbb{C})$ we get a nilpotent orbit $\Oo^\vee$ of $\GC^\vee$ on $\gc^\vee$. Furthermore, the restriction of $\psi$ to $W_\mathbb{R}$ is determined once we specify $\psi(j)$, which must be an element of order two in $\EG$ satisfying:
\begin{enumerate}
\item $\psi(j) \in \text{Cent}_{\GC^\vee}(\psi|_{SL(2,\mathbb{C})})$,
\item $\psi(j) \in \EG - \GC^\vee$. 
\end{enumerate}

Corresponding to $\psi$, let
\begin{equation}
\psi_1 := \psi|_{\SL}: \SL \longrightarrow \GC^\vee, 
\end{equation}
and let 
\begin{equation}
\lambda = \lambda_1 = d\psi_1 \begin{pmatrix} \frac{1}{2} & 0 \\ 0 & -\frac{1}{2}\end{pmatrix} \in \hc^\vee, \quad E_\psi = \psi_1 \begin{pmatrix}
0 & 1 \\
0&0
\end{pmatrix}.
\end{equation}
If we assume that $\lambda$ is integral, which is equivalent to requiring $\Oo^\vee$ to be even, we can construct the parabolic subalgebra $\pc(\lambda)^\vee = \lc(\lambda)^\vee+\nc(\lambda)^\vee\subset \gc^\vee$. Let $\Pp^\vee$ be the $\GC^\vee$-conjugacy class of parabolic subalgebras conjugate to $\pc(\lambda)^\vee$. 
\\
Let $y \in \Ii(\GC^\vee, \gamma^\vee)$ be a representative for a strong real form of $\GC^\vee$ and let $\theta^\vee = \text{Int}(y)$. Let $\gc^\vee = \kc^\vee\oplus \ssc^\vee$ be the Cartan decomposition of $\gc^\vee$ with respect to $\theta^\vee$. In this setting, $E_\psi\in \nc(\lambda)^\vee\cap \ssc^\vee$, and using (\cite{\ABV}, Lemma 27.8), it belongs to the Richardson class (\cite{\ABV}, Proposition 20.4) corresponding to $\Pp^\vee$, denoted as $\Zz(\Pp^\vee)$. 
\\

Let $\Oo^\vee_\lambda$ be the semisimple orbit containing $\lambda$, arising from a homomorphism $\psi_1 : \SL \longrightarrow \GC^\vee$, and let $X(\Oo_\lambda^\vee, \EG):= \{(y', \lambda')\;|\; y'\sim y \;\; \text{and}\;\; \lambda'\in \Oo_\lambda^\vee\}$. We say that a unipotent parameter $\psi'$ is \textit{supported} on $X(\Oo_\lambda^\vee, \EG)$ if the $\lambda' \in \Oo^\vee_\lambda$ ($\lambda'$ is constructed given $\psi'$ as above.). 

\subsection{Parameterizing Theta Forms of a even complex nilpotent orbit.}
Let $\Oo^\vee \subset \gc^\vee$ be a complex even $\GC^\vee$-nilpotent orbit. The goal of this section is to find a ``good" parameterization for the theta-forms of $\Oo^\vee$ defined in (\ref{theta_real_form_of_O}). Using the Kostant-Sekiguchi correspondence, we get a parameterization of the real forms of $\Oo^\vee$ defined in (\ref{real_form_of_O}). 
\\

Let $\{X^\vee, Y^\vee, H^\vee\}$ be the Jacobson-Morosov triple for $\Oo^\vee$, so that $\Oo^\vee = \GC^\vee \cdot X^\vee$. We recall a special case/corollary of (\cite{\ABV}, Theorem 27.10),
\begin{cor}\label{thm_real_forms_of_orbit}
Let $\theta^\vee$ be the Cartan involution of $\GC^\vee$ satisfying $(\GC^\vee)^{\theta^\vee} = \KC^\vee$ and let $\gc^\vee = \kc^\vee\oplus \ssc^\vee$ be the Cartan decompostion. Let $\GR^\vee$ be the real form of $\GC^\vee$ corresponding to the Cartan involution $\theta^\vee$. Furthermore, assume that the semisimple orbit $\Oo_\lambda^\vee$ corresponding to $\lambda$ comes from a homomorphism $\psi_1: \SL \longrightarrow \GC^\vee$ attached to the even nilpotent $\GC^\vee$ orbit $\Oo^\vee$.
Then there is a correspondence between the following sets:
\begin{enumerate}
\item The equivalence classes of unipotent Arthur parameters supported on $X(\Oo_\lambda^\vee, \EG)$.
\item $\KC^\vee$ orbits of parabolic subgroups $\QC^\vee \in \mathcal{P}^\vee=\GC^\vee/\PC^\vee$, where $\PC^\vee$ is a fixed parabolic subgroup of $\GC^\vee$ such that its Levi factor $\LC^\vee$ has Lie algebra $\lc^\vee = \text{Cent}_{\gc^\vee}(\lambda)$ and $\QC^\vee$ satisfies:
\begin{enumerate}
\item $\theta^\vee(\QC)= \QC^\vee $. 
\item Let $\qc^\vee = \lc^\vee+ \nc^\vee$ be the Langlands decomposition of $\mathfrak{q} ^\vee= \text{Lie}(\QC^\vee)$, then 
$$ \nc^\vee \cap \ssc^\vee \cap \Oo^\vee \neq \emptyset$$ 
\end{enumerate}
\item $\KC^\vee$ orbits on $\ssc^\vee \cap \Oo^\vee$. 
\end{enumerate}
\end{cor} 
\begin{proof}
Let $\{X^\vee, Y^\vee, H^\vee\}$ be the Jacobson-Morosov triple for $\Oo^\vee$ and let $\lc^\vee = \text{Cent}_{\gc^\vee}(H^\vee)$. Using (\cite{\CMc}, Corollary 7.1.7), we note that the even nilpotent orbit $\Oo^\vee$ is a Richardson orbit, in fact, it is induced from the trivial obit on $\lc^\vee$. This implies that $\Zz_{\Pp^\vee} = \Oo^\vee$. Using (\cite{\ABV}, Theorem 27.10) in the light of these observations we get the three correspondences. 
\end{proof}

The conditions (1) and (3) in the Corollary above are most intuitive to the reader, yet checking them is not easy. We will use the Atlas Setting and software to work with condition (2): (2a) is has been implemented and can be easily tested for, (2b) is the more difficult one to test, and our method uses representation theory to arrive at an algorithm to test it succesfully in many cases, and point out the cases when it cannot be computed. 
\\

We continue to be in the Atlas Setting as follows: $\GC$ a complex connected reductive algebraic group,  $\GC^\vee$ the dual group. Fix a strong real form $\eta^\vee$ of $\GC^\vee$ and let $\theta^\vee_{\eta^\vee}$ be the corresponding Cartan involution of $\GC^\vee$. Let $\KC^\vee = \text{Cent}_{\GC^\vee}(\eta^\vee)$. Let $\gc^\vee = \kc^\vee \oplus \ssc^\vee$ be the Cartan decomposition of $\gc^\vee$ with respect to $\theta_{\eta^\vee}^\vee$. Furthermore, we choose $\xi$ to be the dual quasisplit strong real form of $\GC$, in the dual inner class of $^\Gamma \GC^\vee$. 
\\

Fix a regular integral infinitesimal character $\gamma \in \hc^*\simeq \hc^\vee$. We are fixing $\xi$, $\eta^\vee$ and $\gamma$, so we will suppress them from the notation. In this setting, we have a block of irreducible representations of $\GR$ at infinitesimal character $\gamma$, $\Bb = \Bb(\xi, \eta^\vee, \gamma) \subset  \Pi(\gc, \KC, \gamma)$. Corresponding to $\Bb$, we have the dual block $\Bb^\vee$ of irreducible $(\gc^\vee, \KC^\vee)$-modules at infinitesimal character $\gamma^\vee$. Since $\gamma^\vee$ is integral, the full Weyl group $\W^\vee$ acts on $\Bb^\vee$. Using the coherent continuation action, $\Bb^\vee$ decomposes into into disjoint $HC$-cells:
\begin{equation}
\mathcal{B}^\vee = \coprod_{ HC-\text{cells}} \Cc^\vee.
\end{equation}
Recall that $\AVC(\Cc^\vee)$ makes sense, since the associated variety remains constant for a fixed cell $\Cc^\vee$. 
\\

Now fix an even nilpotent orbit $\Oo^\vee\subset \gc^\vee$ with Jacobson-Morosov triple $\{X^\vee, Y^\vee, H^\vee\}$ and fix $\lambda=\frac{1}{2}H^\vee$ and let $\lc^\vee = \text{Cent}_{\gc^\vee}(\lambda)$. Let 
\begin{equation}
\Pp^\vee(\lc^\vee) := \{ \text{$\theta^\vee$ - stable parabolic subalgebras of $\gc^\vee$ with Levi-factor $\lc^\vee$}\}\subset \Pp^\vee.
\end{equation}
Then, every $\pc^\vee \in \Pp^\vee(\lc^\vee)$ is conjugate to a parabolic subalgebra the form $\pc(\lambda')^\vee$ for some semisimple $\lambda' \in \lc^\vee$ and the $\theta^\vee$-stable condition comes down to checking that $\theta^\vee(\lambda') = \lambda'$. 

\begin{defn}[Parameter set for theta forms of $\Oo^\vee$]\label{defn_S_O}
Suppose $\eta^\vee$ is a strong real form of $\GC^\vee$ and $\theta^\vee = \text{Int}(\eta^\vee)$ is a corresponding Cartan involution of $\GC^\vee$. Associated to the pair $(\Oo^\vee, \eta^\vee)$ is the set 
\begin{equation}
\Ss(\Oo^\vee, \eta^\vee) := \{\pc^\vee\in \KC^\vee\backslash\Pp^\vee(\lc^\vee) \;|\; \theta^\vee(\pc^\vee) =\pc^\vee, \;\;  \nc^\vee \cap \ssc^\vee \cap \Oo^\vee \neq \emptyset \},
\end{equation} 
where $\pc^\vee = \lc^\vee + \nc^\vee$ is Langlands decomposition of $\pc^\vee$ and $\gc^\vee = \kc^\vee \oplus \ssc^\vee$ is the Cartan decomposition of $\gc^\vee$. In our setting, $\eta^\vee$ is fixed, so we will drop it from the notation, i.e. the parameter set will be denoted as $\Ss(\Oo^\vee)$. 
\end{defn}

As noted in Corollary \ref{thm_real_forms_of_orbit}, we know that $\Ss(\Oo^\vee)$ parameterizes the theta forms of $\Oo^\vee$. Even though the conditions defining $\Ss(\Oo^\vee)$ are explicit, it is not clear how one could actually check them and in the process compute $\Ss(\Oo^\vee)$. The first result of this paper addresses this problem:

\begin{theorem}\label{thm_real_orbits}
Suppose $\eta^\vee$ a strong real form of $\GC^\vee$. Choose $\xi$ to be the dual quasisplit strong real form of $\GC$ corresponding to $\eta^\vee$. Let $\Bb(\xi, \eta^\vee, \gamma)$ and $\Bb^\vee(\xi, \eta^\vee, \gamma^\vee)$ be blocks of representations at regular integral infinitesimal characters $\gamma$, $\gamma^\vee$ repectively. Let $\theta$ and $\theta^\vee$ be Cartan involutions of $\GC$ and $\GC^\vee$ correponding to $\xi$ and $\eta^\vee$. 
\\

Let $\Oo^\vee$ be an even nilpotent orbit in the complexified Lie algebra $\gc^\vee$. Let $\{X^\vee, Y^\vee, H^\vee\}$ be the Jacobson-Morosov triple for $\Oo^\vee$ and let $\lc^\vee = \text{Cent}_{\gc^\vee}(H^\vee)$. Let $\gc^\vee = \kc^\vee \oplus \ssc^\vee$ be the Cartan decomposition of $\gc^\vee$ with respect to $\theta^\vee$. Let $\Ss(\Oo^\vee)$ be set in Definition \ref{defn_S_O}
\\

Then,
\begin{equation}
\Ss(\Oo^\vee) \leftrightarrow \{ \pc^\vee\KC^\vee\backslash\Pp^\vee(\lc^\vee) \;|\;\; \AVC(\Rr_{\pc^\vee}(\chi_{\text{triv}})) = \overline{\Oo^\vee}\},
\end{equation}
there is an algorithm to explicitly compute the latter set in Atlas.
 there is an explicit algorithm to compute the set $\Ss(\Oo^\vee)$. This algorithm is implementable in the Atlas of Lie Groups software. 
\end{theorem}

\begin{proof}
The computation of the set $\Ss(\Oo^\vee)$ involves checking for two conditions:
\begin{enumerate}
\item  We need to know how to check if a given parabolic is theta stable, which is a relatively simple task in Atlas.
\item We need to find a method to check the condition $ \nc^\vee \cap \ssc^\vee \cap \Oo^\vee \neq \emptyset$. This is difficult, to put it mildly, to check by hand. The main idea of this theorem is to replace this condition with something more amenable to computation, in this case to reduce it to computing the complex associated variety.  
\end{enumerate}
Let $\lambda=H^\vee$ be the semisimple element in the Jacobson-Morozov triple for $\Oo^\vee$. Recall that $\Pp^\vee(\lc^\vee) $ is the set of theta-stable parabolics having Levi factor $\lc^\vee = \text{Cent}_{\gc^\vee}(\lambda)$. As a first step we find a description of $\KC^\vee\backslash\Pp^\vee(\lc^\vee)$, the $\KC^\vee$-conjugacy classes of parabolics in $\Pp^\vee(\lc^\vee)$. 
\\
Let $S(\lambda)$ be the set of simple roots of $\gc^\vee$ which are singular on $\lambda$. Then, a $\KC^\vee$-conjugacy class of parabolic subgroups is determined by specifying a parabolic $\QC^\vee$ corresponding to the data $(y, S(\lambda))$, where $y$ is a representative for a $\KC^\vee$-orbit of $\GC^\vee/\BC^\vee$. The parabolic $\QC^\vee(y, S(\lambda))$ is $\theta^\vee$-stable if and only if $\theta^\vee_{y}(\lambda) = \lambda$, where $\theta^\vee_y$ is the Cartan involution on $\GC^\vee$ corresponding to the $\KGB$-element $y$. 
\\
Using this description of parabolics in Atlas, we can compute the set $\KC^\vee\backslash\Pp^\vee(\lc^\vee)$ in Atlas.
\\

Since by definition, $\Ss(\Oo^\vee) \subset \KC^\vee\backslash \Pp^\vee(\lc^\vee)$, our goal will be to to pare down $\Ss'(\Oo^\vee):=\KC^\vee\backslash \Pp^\vee(\lc^\vee)$ to $\Ss(\Oo^\vee)$. To achieve this, we will use the condition second condition defining $\Ss(\Oo^\vee)$, that is $\nc^\vee \cap \ssc^\vee \cap \Oo^\vee \neq \emptyset$, for a $\nc^\vee$ arising as the nilpotent part of the the Langlands decomposition of $\pc^\vee\in \Ss'(\Oo^\vee)$. 
\\

Define the map:
\begin{equation}
\Xi : \KC^\vee\backslash\Pp^\vee(\lc^\vee) \longrightarrow \{ \text{$\KC^\vee$ - orbits on $\ssc^\vee$}\},\quad  \pc^\vee \mapsto \KC^\vee\cdot (\nc^\vee\cap\ssc^\vee).
\end{equation}

Using Theorem \ref{Aq_AV}, we have $\nc^\vee\cap\ssc^\vee$ is open and dense in a single $\KC^\vee$-orbit of $\ssc^\vee$, as a result, the map $\Xi$ is well defined. 
\\

We note the following consequences of Theorem \ref{thm_real_forms_of_orbit}:
\begin{enumerate}
\item The image of $\Xi$ contains all the $\KC^\vee$ - orbits on $\ssc^\vee\cap \Oo^\vee$. 
\item The restriction of $\Xi$ to $\Ss(\Oo^\vee)$ is a bijective correspondence between $\Ss(\Oo^\vee)$ and $\KC^\vee$-orbits on $\ssc^\vee\cap\Oo^\vee$.
\end{enumerate}

Therefore, to compute $\Ss(\Oo^\vee)$, it comes down to checking if $\Xi(\pc^\vee)$ is a $\KC^\vee$-orbit on $\ssc^\vee\cap\Oo^\vee$. 

\begin{prop}\label{prop_S(O)_description}
Suppose we are in the above setting and let $\pc^\vee \in \KC^\vee\backslash\mathcal{P}^\vee(\lc^\vee)$. Then $\Xi(\pc^\vee)$ is a $\KC^\vee$-orbit on $\ssc^\vee\cap\Oo^\vee$ (that is $\pc^\vee\in \Ss(\Oo^\vee)$) if and only if $\AVC(\Rr_{\pc^\vee}(\chi_{\text{triv}})) = \overline{\Oo^\vee}$, where $\chi_{\text{triv}}$ is the trivial character on $\pc^\vee$. 
\end{prop}
\begin{proof}
Suppose $\Xi(\pc^\vee)$ is a $\KC^\vee$-orbit on $\ssc^\vee\cap\Oo^\vee$, then $\nc^\vee\cap\ssc^\vee\cap \Oo^\vee \neq \emptyset$. Let $X^\vee\in \nc^\vee\cap\ssc^\vee\cap \Oo^\vee$ be a generic element, then using Theorem \ref{Aq_AV}, we see that 
\begin{equation*}
\AVT(\mathcal{R}_{\pc^\vee}(\chi_{\text{triv}})) = \overline{\KC^\vee\cdot (\nc^\vee\cap \ssc^\vee)} = \overline{\KC^\vee\cdot X^\vee},
\end{equation*}
therefore  $\Xi(\pc^\vee)$ is by definition $\AVT(\Rr_{\pc^\vee}(\chi_{\text{triv}}))$. 
\\
Furthermore, using the relationship between $\AVC$ and $\AVT$ for a fixed module, we have
\begin{eqnarray*}
 \AVC(\Rr_{\pc^\vee}(\chi_{\text{triv}}))& =& \overline{\GC^\vee \cdot (\KC^\vee\cdot (\nc^\vee\cap \ssc^\vee))}\\
 &=&\overline{\GC^\vee \cdot (\KC^\vee\cdot X^\vee)}\\
 &=& \overline{\GC^\vee \cdot X^\vee}\\
 &=&\overline{\Oo^\vee}
\end{eqnarray*}

This implies that $\AVC(\Rr_{\pc^\vee}(\chi_{\text{triv}})) =  \overline{\Oo^\vee}$ if  $\nc^\vee\cap\ssc^\vee\cap \Oo^\vee \neq \emptyset$.
\\

Now, if $\Xi(\pc^\vee)$ is not a $\KC^\vee$-orbit on $\ssc^\vee\cap\Oo^\vee$, then by Corollary \ref{thm_real_forms_of_orbit}, $\Xi(\pc^\vee) =\AVT(\Rr_{\pc^\vee}(\chi_{\text{triv}}))$ cannot be a theta form of the complex nilpotent orbit $\Oo^\vee$, so that $\AVC(\Rr_{\pc^\vee}(\chi_{\text{triv}})) = \GC^\vee\cdot \AVT(\Rr_{\pc^\vee}(\chi_{\text{triv}})) \neq \overline{\Oo^\vee}$. This completes the proof of the proposition. 
\end{proof}
Given the fact that we can compute $\KC^\vee\backslash\mathcal{P}^\vee(\lc^\vee)$ using Atlas, Proposition \ref{prop_S(O)_description} reduces the computation of $\Ss(\Oo^\vee)$ to the computation of complex associated varieties of all representations in the given block $\Bb^\vee$. It turns out that there are algorithms to take care of this latter step, it is dealt in two cases: 
\begin{enumerate}
\item When $\GC^\vee$ is of classical type. 
\item When $\GC^\vee$ is of exceptional type.
\end{enumerate} 
For Case 1, an algorithm by Noel-Jackson provides the special $\W^\vee$-representation attached to a representation $\pi^\vee$ of $\GR^\vee$ by computing the special $\W^\vee$-representation $\sigma(\Cc^\vee)$ attached to the $HC$-cell $\Cc^\vee$ containing $\pi^\vee$, using the tau-invariants of the representations in $\Cc^\vee$.
\\
We can then apply the Springer correspondence to compute the special nilpotent orbit attached to $\sigma(\Cc^\vee)$, the closure of this special nilpotent orbit is  $\AVC(\pi^\vee))$. Both, the Noel-Jackson algorithm and the Springer correspondences can be implemented as functions in the Atlas software, so that given the block $\Bb^\vee$, one gets an output specifiying $\AVC(\Cc^\vee)$ for every $HC$-cell in $\Bb^\vee$.
\\
For Case 2, we use tables computed by Binegar to find out what the $\AVC(\Cc^\vee)$, for a given $HC$-cell in $\Bb^\vee$. There is an algorithm due to Vogan that would compute the special Weyl group representation of an irreducible representation of a group of exceptional type, work is in progress to write it down in a way that could be implemented in the Atlas software. 
\\

This completes the proof of the theorem. For the reader's convenience, we summarize the algorithm to compute $\Ss(\Oo^\vee)$:
\begin{enumerate}
\item Given $\Oo^\vee$, compute the neutral element $H^\vee$ in the Jacobson-Morozov triple for $\Oo^\vee$ and let $\lambda=H^\vee$. Let $\lc^\vee = \text{Cent}_{\gc^\vee}(\lambda)$. 
\item Compute the set $\Ss'(\Oo^\vee):=\KC^\vee\backslash\Pp^\vee(\lc^\vee)$, which is possible in Atlas.
\item Using Noel-Jackson algorithm or the tables by Birne, compute the $\AVC(\pi^\vee)$ for every $\pi^\vee\in \Bb^\vee$. 
\item To pare down $\Ss'(\Oo^\vee)$ to $\Ss(\Oo^\vee)$, for every $\pc^\vee\in \Ss'(\Oo^\vee)$ compute $\AVC(\Rr_{\pc^\vee}(\chi_{\text{triv}}))$ using previous step. If $\AVC(\Rr_{\pc^\vee}(\chi_{\text{triv}})) = \overline{\Oo^\vee}$ put $\pc^\vee$ into $\Ss(\Oo^\vee)$ else discard it from the list. 
\item Since $\Ss'(\Oo^\vee)$ was a finite set, this algorithm will terminate in finite number of steps and at the end we will be left with exactly $\Ss(\Oo^\vee)$. 
\end{enumerate}
\end{proof}

\subsection{Computing Real Associated Variety}
Continuing in the setting of the last section, now we describe an algorithm to compute the theta-real-associated variety of an irreducible $(\gc^\vee, \KC^\vee)$-module, $\pi^\vee \in \Bb^\vee$, at infinitesimal character $\gamma^\vee$.  Let $\AVC(\pi^\vee) = \overline{\Oo^\vee}$. We specify certain ``good" conditions on $\pi^\vee$, which when satisfied, $\AVT(\pi^\vee)$ can be explicitly computed. 
\\

To begin, fix the nilpotent orbit $\Oo^\vee$ in $\gc^\vee$ and let $\lambda \in H^\vee$ be the semisimple element in the Jacobson-Morozov triple corresponding to $\Oo^\vee$. Let $\pc^\vee(\lambda) = \lc^\vee+\nc(\lambda)^\vee$ be the parabolic subalgebra corresponding to $\lambda$. If we assume that $\Oo^\vee$ is even, the set $\Ss(\Oo^\vee)$ corresponding to the real forms of $\Oo^\vee$ in terms of representatives of $\KC^\vee$-conjugacy classes of $\theta^\vee$-stable parabolics, is computable and the algorithm is described in Section 6.4. The first ``good" condition to compute $\AVT(\pi^\vee)$ is as follows:
\\

\begin{center}
\begin{large} {Condition 1}
\end{large}
\\
\textbf{$\AVC(\pi^\vee)$ is the closure of an even nilpotent orbit $\Oo^\vee$.}
\end{center} 
\vspace{.3in}

Fix a parabolic $\qc^\vee \in \Ss(\Oo^\vee)$ and let $\qc^\vee=\lc^\vee+\nc^\vee$ be its Langlands decomposition. Let $\QC^\vee = \LC^\vee\NC^\vee$ be the corresponding parabolic of $\GC^\vee$, with $\LC^\vee$ the complexification of the real Levi $\LR^\vee$. 
\\
Let $\lambda$ be a one dimensional $(\lc^\vee, \LC^\vee\cap\KC^\vee)$-module, we then cohomologically induce it up to a $(\gc^\vee, \KC^\vee)$-module, denoting this final representation as $\Rr_{\qc^\vee}(\lambda)$. 

\begin{prop}
Suppose $\qc^\vee$ is as above and $\chi_1$ and $\chi_2$ are two one dimensional representations of $\LC^\vee\cap\KC^\vee$, such that $\chi_1|_{(\LC^\vee\cap\KC^\vee)_0} = \chi_2|_{(\LC^\vee\cap\KC^\vee)_0}$, then 
\begin{equation}
\AVT(\Rr_{\qc^\vee}(\chi_1)) =\AVT(\Rr_{\qc^\vee}(\chi_2)).
\end{equation}
\end{prop}
\begin{proof}
Since $\chi_1|_{(\LC^\vee\cap\KC^\vee)_0} = \chi_2|_{(\LC^\vee\cap\KC^\vee)_0}$, we have an equality of derivatives $d\chi_1 = d \chi_2$, let's call this $\lambda_0$. Using Theorem \ref{Aq_AV} and the discussion preceding it, we know that the theta-associated variety $\AVT(\Rr_{\qc^\vee}(\chi_1))$ and $\AVT(\Rr_{\qc^\vee}(\chi_2))$ depend only on $\nc(\lambda_0)$, and hence the equality of the two thetal associated varieties follows. 
\end{proof}

We know how to compute all the real forms of $\Oo^\vee$ in terms of $\theta^\vee$-stable parabolics, using Theorem \ref{thm_real_orbits} denoted as $\Ss(\Oo^\vee)$. Suppose $\Ss(\Oo^\vee) = \{ \pc_1^\vee, \pc_2^\vee, \dots, \pc_r^\vee\}$ and let $\{\LC_1^\vee, \LC_2^\vee, \dots, \LC_r^\vee\}$ be the corresponding the Levi subgroups. Choose one dimensional $(\lc^\vee_i, \LC_i^\vee\cap \KC^\vee)$-modules $\chi_{ij}$ for $j = 0, 1, 2, \cdots, s-1$ where $s = |\LC^\vee/\LC^\vee_0|$ such that the infinitesimal character of $\chi_{ij} = \rho(\LC_i)=\gamma_L$ (in fact, any regular integeral infinitesimal charactar for $\LC_i$ would work, we make this choice so that our induced modules have infinitesimal character $\rho(\GC)$). For a fixed $i$ and for all $j$,  $\AVT(\Rr_{\pc_i^\vee}(\chi_{ij}))$ corresponds to the $\KC^\vee$-orbit corresponding to the same $\pc_i$. That is, the real associated variety remains the same as we vary $j$ but keep $i$ fixed.
\\

Recall that $\mathcal{B}^\vee$ is partitioned into $\W^\vee$-cells, $\Cc^\vee$, given by 
$$ \mathcal{B}^\vee = \coprod \Cc^\vee.$$  

\begin{defn}[``Good" Cells]
In the above setting we say that a $\W^\vee$-cell $\Cc^\vee$ is good if it contains a representation of the form $\Rr_{\pc_i^\vee}(\chi_{ij})$ for some choice of $i$ and $j$. 
\end{defn}
Let $\Cc_{ij}^\vee$ be the set of good cells corresponding to the representations $\Rr_{\pc_i^\vee}(\chi_{ij})$ for $i\in \{ 1, 2, \dots, r\}$ and $j\in \{1, 2, \dots, s\}$. Note that $\AVT(\Cc_{i,j}^\vee)$  is the closure of a single $\KC^\vee$-orbit that corresponds to the $\theta^\vee$-stable parabolic $\pc_i^\vee \in \Ss(\Oo^\vee)$. 
\\

We can now state the second good condition:
\\

\begin{center}
\begin{large} {Condition 2}
\end{large}
\\
\textbf{$\pi^\vee$ lies in a good cell.}
\end{center} 
\vspace{.3in}

\begin{defn}[Good Condition]\label{defn_good_condition}
Suppose $\pi^\vee\in \Bb^\vee$. We say that $\pi^\vee$ satisfies the good condition if Conditions 1 and 2 (above) are both satisfied.  
\end{defn}
We are led to the following theorem:
\begin{theorem}\label{thm_real_ass_var}
Let $\xi$, $\eta^\vee$ be strong real forms of $\GC$, $\GC^\vee$ in the Atlas Setting, Definition \ref{AtlasSetting}. Let $\Bb$ and $\Bb^\vee$ be blocks of representations at regular integral infinitesimal characters $\gamma$, $\gamma^\vee$ repectively. Let $\Oo^\vee$ be a fixed even nilpotent orbit and let $\pi^\vee\in\Bb^\vee$. 
\\
Then, 
\begin{enumerate}
\item  there is an explicit algorithm to compute $\AVC(\pi^\vee)$. 
\item if $\pi^\vee$ satisfies the ``good condition" (that is when Condition 1 and 2 are both satisfied), then there is a algorithm to explicitly compute $\AVT(\pi^\vee)$ as the closure of a single $\KC^\vee$-orbit in $\ssc^\vee\cap\Oo^\vee$.
\item the algorithms described above are implemented into the Atlas Software.
\end{enumerate}  
\end{theorem}
\begin{proof}
We describe the two algorithms mentioned in the theorem. The algorithm to compute the complex associated variety is as follows:
\begin{enumerate}
\item Suppose $\pi^\vee\in\Bb^\vee$. By the decomposition of $\Bb^\vee$ into $HC$-Cells, there must be a cell $\Cc^\vee$, such that $\pi^\vee \in \Cc^\vee$. 
\item When $\GC$ is of classical type, we use the Noel-Jackson algorithm to compute the special Weyl group representation $\sigma(\Cc^\vee)$ attached to $\Cc^\vee$, this algorithm has been implemented in Atlas. If $\GC$ is of Exceptional type, there are tables for the special Weyl groups representations attached to cells, by Birne Binegar for example. 
\item We apply the Springer correspondence (again implemented in Atlas) to $\sigma(\Cc^\vee)$ to get the special nilpotent orbit attached to $\Cc^\vee$, by construction, this is exactly $\AVC(\Cc^\vee)$. 
\item Since the associated variety remains constant on the cell, we have hence computed $\AVC(\pi^\vee)$.
\end{enumerate}
When $\GC$ is of exceptional type, these computations have already been tabulated in literature. We mention that case here only for the sake of completeness.
\\

\begin{enumerate}
\item We compute $\AVC(\pi^\vee)$ using the previous algorithm, this will be a closure of a single nilpotent orbit. It is possible to check if this complex nilpotent orbit is even, if it is even, we have Condition 1 satisfied and denote this nilpotent orbit as $\Oo^\vee$.
\item Since $\Oo^\vee$ is even, we can compute $\Ss(\Oo^\vee)$ corresponding to the block $\Bb^\vee$ as in Theorem \ref{thm_real_orbits}.
\item Now, suppose $\pi^\vee$ is in a good cell, say $\Cc^\vee$. By definition of a good cell, $\Cc^\vee$ must contain a $\Rr_{\pc_i^\vee}(\chi_{ij})$ for a choice of $i$ and $j$. The theta-associated variety of $\Rr_{\pc_i^\vee}(\chi_{ij})$ and hence of $\Cc^\vee$ is the closure of a single theta-form parameterized by $\pc_i\in \Ss(\Oo^\vee)$. 
\end{enumerate}
Therefore if the good condition is satisfied, we can compute the theta-associated variety of $\pi^\vee$ as the closure of the theta-form corresponding to the parabolic $\pc_i \in \Ss(\Oo^\vee)$. This algorithm has been implemented in the Atlas software, so that if you input a representation into the software, we can check if the good condition holds, and if it does, we output the theta-associated variety in terms a parabolic corresponding to a theta-form of $\Oo^\vee$. 
\end{proof}

\section{Special Unipotent Packets for Real Reductive Groups}
We now return to the main goal of this paper, to compute unipotent Arthur packets in the ``good" case, and when things are not ``good", to provide definitive places to look for completing the packets. We provide an algorithm that explicitly computes Atlas/Langlands parameters of representations in these packets, and that has been implemented into the Atlas of Lie groups software.  

\subsection{Special Unipotent Parameters and Packets}
We will be in the Atlas Setting of Definition \ref{AtlasSetting}. Thas is:
\\
Let $(\GC, \gamma)$ be a basic data and $(\GC^\vee, \gamma^\vee)$ be the corresponding dual basic data. Let $(\BC, \HC, \{X_\as\})$ be a fixed pinning for $\GC$. Let $\xi$ be a strong real form of $\GC$ in the inner class of $\gamma$ and let $\eta^\vee$ be a strong real form for $\GC^\vee$ in the dual inner class given by $\gamma^\vee$. Corresponding to $\xi$ and $\eta^\vee$, let $\theta_\xi = \text{Int}(\xi)$, and $\theta_{\eta^\vee} = \text{Int}(\eta^\vee)$ be Cartan involutions of $\GC$ and $\GC^\vee$ respectively with maximal complex subgroups $\KC_\xi$ and $\KC^\vee_{\eta^\vee}$ respectively. 
\\
Let $\delta$ be a regular integral infinitesimal character for $\GC$ and let $\Bb(\xi, \eta^\vee, \delta)$ be the block of irreducible $(\gc, \KC_\xi)$-modules at infinitesimal character $\delta$ specified by the pair of strong real forms $(\xi, \eta^\vee)$. Also, $\Bb^\vee = \Bb(\eta^\vee, \xi,  \delta^\vee)$ is the corresponding dual block of irreducible $(\gc^\vee, \KC_{\eta^\vee}^\vee)$-modules.
\\
\begin{defn}[Block at Singular infinitesimal character]
Suppose $\lambda \in \delta + X^*(\HC)$, then by a block at (possibly singular infinitesimal character) $\lambda$ we will mean the translation of the block at regular integral infinitesimal character at $\delta$ to the infinitesimal character $\lambda$, that is 
\begin{equation}
\Bb(\lambda) = \Bb(\xi, \eta^\vee,  \lambda) := T_{\delta}^\lambda (\Bb(\xi, \eta^\vee,  \delta)).
\end{equation}
\end{defn}
Note that $\Bb(\lambda)$, depends on the choice of a regular integral $\delta\in X^*(\HC)$ such that $\lambda\in \delta+X^*(\HC)$. It is a property of translation functors that the block $\Bb(\lambda)$ does not depend on the choice of $\delta$. 
\\

Fix a unipotent Arthur parameter $\psi$, and let $\phi_\psi$ be the corresponding Langlands parameter with datat $(y, \lambda)$. We recall that the pair $(y, \lambda)$ satisfies:
\begin{enumerate}
\item Let $\psi_0$ be the tempered Langlands parameter corresponding to the restriction of $\psi$ to $\RWG$. Let $(y_0, \lambda_0)$ be the data corresponding to the parameter $\psi_0$. 
\item Let $\psi_1$ be the restriction of $\psi$ to $\SL$. Define:
$$ y_1 = \psi \begin{pmatrix} i &0 \\ 0 &-i \end{pmatrix}, \quad \lambda_1 = d\psi_1 \begin{pmatrix} 1/2 &0\\0&-1/2\end{pmatrix}.$$
Then the Langlands parameter  $\phi_\psi$ corresponding to $\psi$ is given by $(y, \lambda)$ where, $y = y_0y_1$ and $\lambda = \lambda_0+\lambda_1$. 
\end{enumerate}

Recall that $y$ must satisfy: $y^2 = \text{exp}(2\pi i\lambda)$.  We can attach the nilpotent element  $E_\psi := d\psi_1 \begin{pmatrix}0& 1 \\ 0&0\end{pmatrix}$ to $\psi$ . The element $E_\psi \in \nc(\lambda)^\vee\cap \ssc^\vee$. Let $\Oo^\vee = \GC^\vee\cdot E_\psi$ , then $\Oo^\vee$ is a nilpotent orbit for $\GC^\vee$ if and only if $\lambda$ is integral, and in this case $\Oo^\vee$ is even. Let  $\{X^\vee, H^\vee, Y^\vee\}$ be the Jacobson-Morozov triple corresponding to $\Oo^\vee$, and now define
\begin{equation}
\lambda(\Oo^\vee) := \frac{1}{2} H^\vee.
\end{equation}

\begin{defn}[Weak Unipotent Arthur Packet]
Let $\Oo^\vee$ be a dual even complex nilpotend orbit. Choose $\delta$ such that $\lambda(\Oo^\vee)\in X^*(\HC)$. The weak unipotent packet corresponding to the triple $(\xi, \eta^\vee, \Oo^\vee)$ is the set 
\begin{equation}
\Pi_{\text{weak}}^u(\xi, \eta^\vee, \Oo^\vee) := \{ \pi \in \Bb(\lambda(\Oo^\vee)):= T_\delta^{\lambda(\Oo^\vee)}(\Bb(\delta))\;|\; \AVC(\pi^\vee) = \overline{\Oo^\vee}\}. 
\end{equation}
\end{defn}

An easy consequence of the definition weak unipotent packets is the fact that two weak unipotent Arthur packets are either equal or disjoint.
\\

We can construct the parabolic subalgebra $\pc^\vee = \lc(\lambda)^\vee+\nc(\lambda)^\vee$ and define $\Pp^\vee$ to be the conjugacy class of parabolic subalgebras conjugate to $\pc^\vee$. In this setting $E_\psi\in \Zz_{\Pp^\vee}$, the Richardson orbit corresponding to $\Pp^\vee$. Since $\Oo^\vee$ is even, $\Zz_{\Pp^\vee} = \Oo^\vee$. This implies that $E_\psi \in \nc(\lambda)^\vee\cap\ssc^\vee\cap \Oo^\vee$, as a result, we can find a $\KC^\vee$ orbit on $\ssc^\vee$  (this is exactly $\KC^\vee\cdot E_\psi$) corresponding to the Arthur parameter $\psi$, call this orbit $\Oo^\vee_{\KC^\vee}$. The map 
\begin{equation}
\psi \mapsto \KC^\vee\cdot E_\psi,
\end{equation}
defines a bijection between unipotent Arthur packets corresponding to unipotent Arthur parameters supported on $X_j(\Oo_\lambda^\vee, \EG)$ and the theta real forms of $\Oo^\vee$ in the block $\Bb^\vee(\lambda(\Oo^\vee))$. Note that $\GC^\vee \cdot E_\psi = \Oo^\vee$. We now define a special unipotent Arthur packet.

\begin{defn}[Special Unipotent Arthur Packet] 
The special unipotent Arthur packet corresponding to the tuple $(\xi, \eta^\vee, \Oo^\vee_{\KC^\vee})$ is the set 
\begin{equation}
\Pi^u(\xi, \eta^\vee, \Oo_{\KC^\vee}^\vee) := \{ \pi \in\Pi_{\text{weak}}^u(\xi, \eta^\vee, \Oo^\vee) \;|\; \overline{\Oo_{\KC^\vee}^\vee}\subset \AVT(\pi^\vee)\}. 
\end{equation}
\end{defn}
The theta associated variety of an irreducible representation need not necesarily be the closure of a single orbit, as a result, we can only hope for an inclusion of $\Oo^\vee_{\KC^\vee}$ inside $\AVT(\pi^\vee)$ as a result two distinct unipotent Arthur packets need not necessarilly be disjoint.

\subsection{Computing Special Unipotent Packets}
Continuing with the definitions of unipotent packets, we now proceed to compute them. Even though the packets are explicitly defined, the computation of it's actual contents is difficult. The most difficult step in this process is the computation of the invariants $\AVC(\pi)$ and $\AVT(\pi)$. 
\\
There is no algorithm that computes the contents of the special unipotent packets. The results on the computation of the invariants $\AVC(\pi)$ and $\AVT(\pi)$ (Theorem \ref{thm_real_ass_var}) in the earlier sections provide us the tools to explicitly compute these packets. As a result, we now show how to completely compute weak unipotent Arthur packets. The algorithm to compute special unipotent Arthur packets, does not provide complete answers in all cases, however, when the packets remain incomplete, we can provide an exact list of representations that could possibly complete these packets. We use the Atlas software to implement these algorithms to get explicit Langlands parameters for representations in the given unipotent packet. 
\\

Fix a dual even complex nilpotent orbit $\Oo^\vee$. Let $\lambda:=\lambda(\Oo^\vee)=\frac{1}{2}H^\vee$, where $H^\vee$ is the semisimple element in the Jacobson-Morozov triple. We want to compute 
\begin{equation}
\Pi_{\text{weak}}^u(\xi, \eta^\vee, \Oo^\vee) := \{ \pi \in \Bb(\lambda)\;|\; \AVC(\pi^\vee) = \overline{\Oo^\vee}\}. 
\end{equation}

Fix a regular and integral $\delta\in X^*(\HC)$, such that $\lambda\in \delta + X^*(\HC)$. Let $\Bb(\xi, \eta^\vee, \delta)$ be a block of irreducible $(\gc, \KC)$-modules at regular integral infnitesimal character $\delta$, note that $\eta^\vee$ satisfies $(\eta^\vee)^2 = e^{2\pi i \lambda}$. Let $\Bb^\vee$ be the dual block, a block of irreducible $(\gc^\vee, \KC^\vee)$-modules at infnitesimal character $\delta^\vee$. Using Vogan duality, we note that there is a bijection between $\Bb(\gamma)$ and $\Bb^\vee(\gamma^\vee)$. If $\pi\in \Bb(\gamma)$, it's Vogan dual will be denoted as $\pi^\vee \in \Bb^\vee(\gamma^\vee)$.

This leads us to the first main result of this section:
\begin{theorem}
Let $\Oo^\vee$ be an even nilpotent orbit in $\gc^\vee$. Suppose we are in the setting described above, then $\Pi_{\text{weak}}^u(\xi, \eta^\vee, \Oo^\vee)$ can be completely and explicitly computed. This algorithm can be implemented into the Atlas Software.
\end{theorem}
\begin{proof}
We will prove this result in a series of steps as follows:
\begin{enumerate}
\item Recall that the dual block is a disjoint union of $HC$-cells
$$ \mathcal{B}^\vee(\delta^\vee) = \coprod \Cc^\vee.$$  
\item Given a $\W^\vee$-cell $\Cc^\vee$ and  for any $\pi^\vee \in \Cc^\vee$, by Theorem \ref{thm_real_ass_var} part (a), we know how to compute  $\AVC(\pi^\vee)$. Since the associated variety remains constant on $\Cc^\vee$, this lets us compute $\AVC(\Cc^\vee)$. 
\item Let $\CC^\vee(\Oo^\vee)$ be the set of all cells $\Cc^\vee$ satisfying $\AVC(\Cc^\vee) = \overline{\Oo^\vee}$. For every cell $\Cc^\vee$, we use Vogan-duality to compute the dual cell $\Cc$ and put this cell into the set $\CC(\Oo^\vee)$, so that $\Cc(\Oo^\vee)$ is set of $HC$-cells for $\Bb(\delta)$ such that the dual cell $\Cc^\vee$ has complex associated variety $\overline{\Oo^\vee}$.
\item The representations in the cells $\Cc\in \CC(\Oo^\vee)$ all have dual complex associated variety $\overline{\Oo^\vee}$, and, have infinitesimal character $\delta$. To get representations at infinitesimal character $\lambda:=\lambda(\Oo^\vee)$, we apply the translation functor $T_\delta^\lambda$. Since we chose $\delta$ such that $\lambda\in \delta+X^*(\HC)$, the application of the translation functor is valid.
\item Therefore, the computation of the weak unipotent Arthur packets is the set: 
\begin{equation}
\Pi_{\text{weak}}^u(\xi, \eta^\vee, \Oo^\vee) = \coprod_{\Cc\in \CC(\Oo^\vee)} T_\delta^\lambda(\Cc).
\end{equation}
\end{enumerate}
To implement this algorithm in the Atlas software, in addition to using in built functions (for induction and translation functors), we have:
\begin{enumerate}
\item implemented the algorithm to compute $H^\vee$, the semisimple element in the Jacobson-Morozov triple corresponding to $\Oo^\vee$.
\item implemented the Noel-Jackson algorithm to compute the special Weyl group representation when $\GC$ is of classical type. In the case when $\GC$ is of exceptional type we hard code the special nilpotent orbit attached to a cell. 
\item implemented the Springer correspondence to compute the special nilpotent orbit given the special $\W$-representation, when $\GC$ is of classical type. 
\item implemented Vogan-duality to compute a dual cell.  
\item each of these functions have combined so that if one inputs the pair $(\Bb, \Oo^\vee)$ we output the set $\Pi_{\text{weak}}^u(\xi, \eta^\vee, \Oo^\vee)$ in terms of explicit Langlands parameters. 
\end{enumerate} 
\end{proof}

We move to computing special unipotent Arthur packets. Let $\Oo^\vee$ be a fixed dual even complex nilpotent orbit for $\GC$. Let $\lambda:=\lambda(\Oo^\vee)$ be the infinitesimal character attached to $\Oo^\vee$. Let $\xi$ be a strong real form of $\GC$. Choose $\delta\in X^*(\HC)$ so that $\lambda\in \delta+X^*(\HC)$. Let $\eta^\vee$ be a strong real form for $\GC^\vee$ such that $(\eta^\vee)^2=e^{2\pi i \lambda}$ and $\eta^\vee$ is in the dual quasisplit inner class for $\GC$. Let $\Bb(\xi, \eta^\vee, \delta)$ be a block for the strong real form $\xi$ of $\GC$ at infinitesimal character $\delta$. Let $\Bb^\vee(\delta^\vee)$ be the corresponding dual block. 
\\

Given the complex nilpotent orbit $\Oo^\vee$, we have a set of $\KC^\vee$-conjugacy classes of $\theta^\vee$-stable ($\theta^\vee=\text{Int}(\eta^\vee)$) parabolic subgalgebras in $\gc^\vee$ parameterizing the theta forms of $\Oo^\vee$ in the block $\Bb^\vee(\delta^\vee)$, denoted as $\Ss(\Oo^\vee)$ and computed in Theorem \ref{thm_real_orbits}.
\\
Let $\Ss(\Oo^\vee) = \{\pc^\vee_1, \pc_2^\vee, \dots, \pc_r^\vee\}$ and suppose for $i=1, 2, \dots, r$, let $\pc_i^\vee = \lc^\vee_i+\nc^\vee_i$ be the Langlands decomposition. Let $\{\Oo^\vee_{\KC^\vee,1}, \Oo^\vee_{\KC^\vee, 2}, \dots, \Oo^\vee_{\KC^\vee,r}\}$ be the set of theta real forms of $\Oo^\vee$ corresponding to the ordered set $\Ss(\Oo^\vee)$.  
\\

For a fixed $i$, let $s_i$ be the number of connected components of the real Levi subgroup, $\LC^\vee_i$, corresponding to $\lc_i$. For $j=1, 2, \dots, s_i$, let $\chi_{ij}$ be a character on $\LC^\vee_i$ such that $d\chi_{ij} = \delta_{\LC^\vee}^\vee$ for all $j=1, 2, \dots, s_i$. 
\\
Corresponding to each $\chi_{ij}$, let $\Cc^\vee_{i,j}$ be the $HC$-cell in $\Bb^\vee(\delta^\vee)$ containing $\Rr_{\pc_i^\vee}(\chi_{ij})$. Following the proof of Theorem \ref{thm_real_ass_var}, recall that for a fixed $i$, $\AVT(\Cc^\vee_{ij}) = \overline{\Oo^\vee_{\KC^\vee, i}}$. We define 
\begin{equation}
\CC^\vee(\Oo^\vee_{\KC^\vee, i}) := \{\Cc^\vee_{ij}\;| \; j=1, 2, \dots s_i\} \subset \CC^\vee(\Oo^\vee).
\end{equation}
Note that every cell $\Cc^\vee \in \CC^\vee(\Oo^\vee_{\KC^\vee, i})$ satisfies $\AVT(\Cc^\vee) = \overline {\Oo^\vee_{\KC^\vee, i}}$. Let 
\begin{equation}
\CC(\Oo^\vee_{\KC^\vee, i}) := \{\text{the dual cell of $\Cc^\vee$, for every $\Cc^\vee\in\CC^\vee(\Oo^\vee_{\KC^\vee, i})$} \} \subset \CC(\Oo^\vee).
\end{equation}
Now, let 
\begin{equation}
\Pi^u_{\text{icp}}(\xi, \eta^\vee, \Oo_{\KC^\vee, i}^\vee) := \coprod_{\Cc\in \CC(\Oo^\vee_{\KC^\vee, i}) } T_\delta^\lambda(\Cc).
\end{equation}
Note that $\Pi^u_{\text{icp}}(\xi, \eta^\vee, \Oo_{\KC^\vee}^\vee) \subset \Pi^u(\xi, \eta^\vee, \Oo_{\KC^\vee}^\vee)$. Every representation $\pi\in\Pi^u_{\text{icp}}(\xi, \eta^\vee, \Oo_{\KC^\vee}^\vee)$ is such that $\AVT(\pi^\vee)$ is the closure of a single theta form $\Oo^\vee_{\KC^\vee, i}$ of $\Oo^\vee$, and hence irreducible. 
\\
It can be the case that there is a representation $\pi\in \Bb(\lambda)$ such that $\AVT(\pi^\vee)$ is reducible and that $\Oo^\vee_{\KC^\vee, i}$ is just one of the components, then $\pi$ must belong to $\Pi^u(\xi, \eta^\vee, \Oo_{\KC^\vee, i}^\vee)$, however such a $\pi$ cannot belong to $\Pi^u_{\text{icp}}(\xi, \eta^\vee, \Oo_{\KC^\vee, i}^\vee)$. That is why we use the subscript ``icp" which stands for ``incomplete packet". 
\\

Now, recall that $\CC(\Oo^\vee)$ is the set of all cells $\Cc\in \Bb(\delta)$ such that $\AVC(\Cc^\vee) =\overline{\Oo^\vee}$. Let 
\begin{equation}
\CC_{\KC^\vee}(\Oo^\vee) = \coprod_{i=1}^r \CC(\Oo^\vee_{\KC^\vee, i}) \subset \CC(\Oo^\vee).
\end{equation}

\begin{defn}[Good Condition for Unipotence]
We will say that the good condition for unipotence is satisfied if $\CC(\Oo^\vee) = \CC_{\KC^\vee}(\Oo^\vee)$. 
\end{defn}
When the good condition for unipotence is satisfied, all the unipotent Arthur packets $\Pi^u_{\text{icp}}(\xi, \eta^\vee, \Oo_{\KC^\vee, i}^\vee) = \Pi^u(\xi, \eta^\vee, \Oo_{\KC^\vee, i}^\vee)$ for all $i=1, 2, \dots, r$. Furthermore, in this case, two unipotent Arthur packets are either disjoint or equal. 
\\

There are cases when the good condition for unipotence is not satisfied. This mostly has to do with the fact that there is no clear understanding about the real associated variety of some $HC$-cell $\Cc^\vee$ in $\Bb^\vee$, in this case $\AVT(\pi^\vee)$ is likely reducible or if it is irreducible $\pi^\vee$ does not belong to any of cells $\Cc^\vee \in \CC^\vee(\Oo{\KC^\vee, i}^\vee)$ for any $i$. 
Let 
\begin{equation}
\Cc_{\text{mis}}(\Oo^\vee) = \CC(\Oo^\vee) - \CC_{\KC^\vee}(\Oo^\vee).
\end{equation}
Then testing for the good condition for unipotence is equivalent to checking if $\CC_{\text{mis}}(\Oo^\vee)$ is empty. Finally let 

\begin{equation}
\Pi^u_{\text{mis}}(\xi, \eta^\vee, \Oo^\vee) := \coprod_{\Cc\in \CC_{\text{mis}}(\Oo^\vee) }T_\gamma^\lambda(\Cc) \subset \Pi^u_{\text{weak}}(\xi, \eta^\vee, \Oo^\vee) .
\end{equation}
The set $\Pi^u_{\text{mis}}(\xi, \eta^\vee, \Oo^\vee)$ is exactly the set of representations, a subset of which when added to $\Pi^u_{\text{ic}}(\xi, \eta^\vee, \Oo_{\KC^\vee, i}^\vee)$, one gets the complete unipotent Arthur packet $\Pi^u(\xi, \eta^\vee, \Oo_{\KC^\vee, i}^\vee)$. For this reason we use the subscript ``mis" which stands for ``missing representations". It is not immediately clear what subset of $\Pi^u_{\text{miss}}(\xi, \eta^\vee, \Oo^\vee)$ can be added to $\Pi^u_{\text{ic}}(\xi, \eta^\vee, \Oo_{\KC^\vee, i}^\vee)$ to get a complete unipotent Arthur packet. In ongoing work with Jeffrey Adams, we explore some ideas about stable characters to achieve this completion in some cases. We summarize the above discussion in the following theorem 

\begin{theorem}\label{thm_main_thm}
Let $\Oo^\vee$ be an even nilpotent orbit in $\gc^\vee$. Let $\xi$ be a strong real form of $\GC$ and let $\eta^\vee$ a strong real form of $\GC^\vee$, and $\delta$ a integral regular infinitesimal character for $\GC$ be chosen such that $(\eta^\vee)^2 = \text{exp}(2\pi i\lambda(\Oo^\vee))$ and $\lambda(\Oo^\vee) \in \gamma+X^*(\HC)$. Let $\Bb(\delta):=\Bb(\xi, \eta^\vee, \delta)$ be the block of $(\gc, \KC_\xi)$-modules and $\Bb^\vee(\delta^\vee)$ the corresponding dual block. Let $\{\Oo^\vee_{\KC^\vee,1}, \Oo^\vee_{\KC^\vee, 2}, \dots, \Oo^\vee_{\KC^\vee,r}\}$ be the theta real forms of $\Oo^\vee$ in the block $\Bb^\vee(\delta^\vee)$. 
\begin{enumerate}
\item Suppose the good condition for unipotence is satisfied, then $\Pi^u (\xi,\eta^\vee, \Oo^\vee_{\KC^\vee, i})$ can be computed for all $i=1, 2, \dots r$. This computation can be implemented in Atlas to compute the explicit Langlands parameters of representations in these complete unipotent packets. 
\item Suppose the good condition for unipotence is not satisfied, then for each $i=1, 2, \dots, r$, we can compute a set 
\begin{equation*}
\Pi^u_{\text{icp}} (\xi, \eta^\vee, \Oo_{\KC^\vee, i}^\vee) \subset \Pi^u (\xi, \eta^\vee,  \Oo_{\KC^\vee, i}^\vee),
\end{equation*}
 and, a set 
 \begin{equation*}
 \Pi^u_{\text{mis}} (\xi, \eta^\vee, \Oo^\vee) \subset \Pi^u_{\text{weak}} (\xi, \eta^\vee, \Oo^\vee) 
\end{equation*}
such that 
\begin{equation*}
\Pi^u(\xi, \eta^\vee \Oo_{\KC^\vee, i}^\vee) - \Pi^u_{\text{icp}} (\xi, \eta^\vee, \Oo_{\KC^\vee}^\vee) \subset  \Pi^u_{\text{mis}} (\xi,\eta^\vee, \Oo^\vee),
\end{equation*}
for each $i=1, 2, \dots, r$. 
For $i=1, 2, \dots, r$, we have the following inclusions:
\begin{equation*}
 \Pi^u_{\text{icp}}(\Oo_{\KC^\vee, i}^\vee)  \subset \Pi^u(\Oo_{\KC^\vee, i}^\vee) \subset \Pi^u_{\text{ic}}(\Oo_{\KC^\vee, i}^\vee) \bigsqcup  \Pi^u_{\text{mis}} (\Oo^\vee) \subset \Pi^u_{\text{weak}} (\Oo_\mathbb{R}^\vee),
\end{equation*}
such that, except $\Pi^u(\xi, \eta^\vee, \Oo_{\KC^\vee,i}^\vee)$, all the other sets are completely and explicitly computable in Atlas. 
\end{enumerate} 
\end{theorem}

\section{An Application}
Recall that Theorem \ref{thm_real_ass_var} computes the real associated variety only when the `good condition' given in Definition \ref{defn_good_condition} is satisfied. 
\\
Here are two possiblities of how the good condition might fail to be true:
\begin{enumerate}
\item The cell $\Cc^\vee$ contains a $\Rr_\qc(\lambda)$, but $\qc$ is not conjugate to any of the parabolics in $\Ss(\Oo^\vee)$. 
\item The cell $\Cc^\vee$ does not contain a cohomologically induced module of the type $\Rr_\qc(\lambda)$ for any choice of theta-stable data $(\qc, \lambda)$. 
\end{enumerate}
In case of (1), we know that the associated variety is definitely irreducible and therefore has to be one of the theta-forms corresponding to a parabolic in $\Ss(\Oo^\vee)$. It is possible that such a scenario does not arise, but at this point we don't know how to prove otherwise. 
\\

In case of (2), we will use Theorem \ref{thm_main_thm} to try to figure out the the associated variety. The main tool in this application is the stable sum formula for unipotent packets which we now state:

\begin{theorem}[Theorem 22.7, \cite{\ABV}]
Suppose we are in setting of previous section, that is the Atlas Setting, let $\Oo^\vee$ be a dual nilpotent orbit and let $\Oo^\vee_{\KC^\vee, i}$ be one of its theta-forms. Let $\Pi^u(\xi, \eta^\vee, \Oo_{\KC^\vee, i}^\vee) := \Pi_u(\Oo^\vee_{\KC, i})$. Then corresponding to $\Oo^\vee_{\KC, i}$, is a strongly stable virtual character given by
\begin{equation}
\eta(\Oo^\vee_{\KC, i}) = \sum_{\pi \in \Pi_u(\Oo^\vee_{\KC, i})} a(\pi)\pi,
\end{equation}
the coefficients $a(\pi)$ are explicitly determined, and are non-zero.
\end{theorem}

Since all the coefficients $a(\pi)$ are non-zero, given a complete Arthur packet, we should be able to compute $\eta(\Oo^\vee_{\KC, i})$. 
\\
Alternately, if we start with a subset of an Arthur packet which does not have a stable sum of virtual characters and we inductively add a representation, from finite set, to this subset checking for stable sums at each step, then, in this scheme, suppose we did not find a stable sum at stage $n$, and we find a stable sum with all non-zero coefficients at stage $n+1$. This would imply that adding these $n+1$ representations to the subset we started with gives us the complete Arthur packet or a better approximation to the Arthur packet than the original subset. In this setup, $ \Pi^u_{\text{icp}}(\Oo_{\KC^\vee, i}^\vee)$ is the subset we want to start with and $\Pi^u_{\text{mis}} (\Oo^\vee)$ is the set from which we add representations. 
\\
Suppose we started out with $\Pi^u_{\text{icp}}(\Oo_{\KC^\vee, j}^\vee)$, $i\neq j$ and repeated the same process as above to compute the unipotent packet $\Pi^u_{\text{icp}}(\Oo_{\KC^\vee, j}^\vee)$, then the representations of $\Pi^u_{\text{mis}}(\Oo^\vee)$ that are in both $\Pi^u_{\text{icp}}(\Oo_{\KC^\vee, i}^\vee)$ and $\Pi^u_{\text{icp}}(\Oo_{\KC^\vee, j}^\vee)$ would contain $ \Oo_{\KC^\vee, i}^\vee$ and $\Oo_{\KC^\vee, j}^\vee)$ in their theta-associated variety, proving that the associated variety is reducible. As we vary over all the theta-forms, we end up computing the associated varieties of all the representations in $\Pi^u_{\text{mis}}(\Oo^\vee)$. 
\\
Work on this is still in progress, a crucial component is the implementation of the computations of stable sum formulas that have been implemented into the Atlas software by Adams. 


\bibliography{MasterRef}

\end{document}